\documentclass[12pt,a4paper,dvipsnames]{article}

\usepackage{hyperref}
\usepackage{pgf}
\usepackage[utf8]{inputenc}
\usepackage{amsfonts}
\usepackage{amssymb}
\usepackage{amsthm}
\usepackage{fixmath}

\usepackage{dsfont}
\usepackage[linesnumbered]{algorithm2e}
\usepackage{setspace}
\usepackage{extarrows}
\usepackage[makeroom]{cancel}
\usepackage{graphicx}
\usepackage{multicol}
\PassOptionsToPackage{dvipsnames}{xcolor}
\usepackage{tikz,xcolor}

\usetikzlibrary{shapes,arrows}
\usetikzlibrary{hobby}

\usepackage[makeroom]{cancel}

\usepackage[english]{babel}		
\usepackage[T1]{fontenc}			
\usepackage{color}         			
\usepackage{listings}					

\usepackage{geometry}
\newtheorem{theorem}{Theorem}[section]
\newtheorem{lemma}[theorem]{Lemma}
\newtheorem{proposition}[theorem]{Proposition}

\newtheorem{remark}[theorem]{Remark}
\newtheorem{question}{Question}[section]
\newtheorem{conjecture}[question]{Conjecture}

\renewcommand{\P}{\mathbb{P}}
\newcommand{\TV}[1]{{\lVert #1 \rVert}_{\normalfont
\text{TV}}}
\newcommand{\N}{\mathbb{N}}
\newcommand{\R}{\mathbb{R}}
\newcommand{\E}{\mathbb{E}}
\newcommand{\Z}{\mathbb{Z}}

\newcommand\abs[1]{\left| #1\right|}

\title{Mixing times for the simple exclusion process in ballistic random environment}
\author{ Dominik Schmid $^{*}$ }
\date{\today}

\begin{document}

\maketitle

\begin{abstract}
We consider the exclusion process on segments of the integers in a site-dependent random environment. We assume to be in the ballistic regime in which a single particle has positive linear speed. Our goal is to study the mixing time of the exclusion process when the number of particles is linear in the size of the segment. We investigate the order of the mixing time depending on the support of the environment distribution. In particular, we prove for nestling environments that the order of the mixing time is different than in the case of a single particle.
\end{abstract}
\phantom{.} \hspace{1.13cm} \textbf{Keywords:} Exclusion process, mixing time, random environment \\
\phantom{.} \hspace{0.87cm} \textbf{AMS 2000 subject classification:} 60K37, 60J27 
\let\thefootnote\relax\footnotetext{* \textit{Technische Universität München, Germany. E-Mail}: \nolinkurl{dominik.schmid@tum.de}}

\section{Introduction}

The exclusion process is one of the most studied examples of an interacting particle system. Intuitively, it can be described in the following way: Suppose that we are given a graph and a set of indistinguishable particles, which we initially place on distinct sites of the graph. Each particle independently performs a random walk on the graph. If a particle would move to a site, which is already occupied by another particle, then the move is suppressed. A variety of situations such as cars in a traffic jam or molecules in a low-density gas can be modeled by the exclusion process. For a general introduction to the exclusion process we refer to Liggett \cite{liggett1999stochastic}. \\

In the following, we assume that the underlying graph is a segment of the integers. We call the resulting process the simple exclusion process. Seen as an ergodic Markov process, our goal is to understand the speed of convergence towards the stationary distribution. More precisely, we are interested in understanding the mixing times. A comprehensive introduction to mixing times can be found in the book of Levin, Peres and Wilmer \cite{MCMT} which also treats the case of the simple exclusion process with constant transition rates. In this paper, we consider the case, where the transition rates of the simple exclusion process are chosen i.i.d. according to some fixed distribution.
\newpage

\subsection{The model}

First, we define the simple exclusion process on finite boxes of $\mathbb{Z}$ in fixed environment. The \textbf{simple exclusion process} in environment $\omega=\{ \omega(x)\}_{x \in \{1,\dots, N\}}$ on a segment of size $N$ with $k$ particles is a Feller process $(\eta_t)_{t \geq 0}$ with state space $\Omega_{N,k}$ for
\begin{equation*}
\Omega_{N,k} := \left\lbrace \eta \in \lbrace 0,1 \rbrace^{N} \colon \sum_{x=1}^{N} \eta(x)= k\right\rbrace
\end{equation*}
and generator
\begin{align*}
\mathcal{L}f(\eta) &= \sum_{x =1}^{N-1} \omega(x) \ \eta(x)(1-\eta(x+1))\left[ f(\eta^{x,x+1})-f(\eta) \right] \nonumber \\
&+ \sum_{x =2}^{N} \left(1-\omega(x)\right) \ \eta(x)(1-\eta(x-1))\left[ f(\eta^{x,x-1})-f(\eta) \right]  
\end{align*}
where $\omega(x) \in (0,1]$ for all $x \in \lbrace 1,\dots,N\rbrace$. Here, $\eta^{x,y}$ denotes the configuration where we exchange the values at positions $x$ and $y$ in $\eta$. For a configuration $\eta$, we say that a site $x$ is \textbf{occupied} by a particle if $\eta(x)=1$ and \textbf{vacant} otherwise. A particle at a vertex $x$ moves to the right at rate $\omega(x)$ and to the left at rate $1-\omega(x)$ whenever the target is a vacant site. For the exclusion process on $\Omega_{N,k}$ in a random environment, we  choose the transition probabilities $\lbrace \omega(x) \rbrace_{x \in \lbrace 1,\dots,N\rbrace}$ to be i.i.d. according to some probability distribution on $(0,1]$ and denote the law of the environment by $\mathbb{P}$. Since $\omega(x) \in (0,1]$ for all $x \in \lbrace 1,\dots,N\rbrace$, the simple exclusion process has a unique essential class and so a unique stationary distribution $\pi^N_{\omega}$. 
We denote the quenched law of the exclusion process in a fixed environment $\omega$ with starting distribution $\lambda$ by $P_{\omega, \lambda}$. If $\lambda$ is the Dirac measure on some configuration $\psi \in \Omega_{N,k}$, we will write $P_{\omega, \psi}$. Define the (\textbf{quenched}) $\mathbold\varepsilon$\textbf{-mixing time} of the exclusion process $(\eta_t)_{t \geq 0}$ to be 
\begin{equation*}
t^{\omega,N}_{\text{\normalfont mix}}(\varepsilon) := \inf\left\lbrace t\geq 0 \ \colon \max_{\psi \in \Omega_{N,k}} \TV{P_{\omega,\psi}\left( \eta_t \in \cdot \ \right) - \pi^N_{\omega}} < \varepsilon \right\rbrace 
\end{equation*} for $\varepsilon \in (0,1)$, where $\TV{ \ \cdot \ }$ denotes the total-variation distance between two probability measures. We will refer to $t^N_{\text{\normalfont mix}}:=t^{\omega,N}_{\text{\normalfont mix}}(\frac{1}{4})$ simply as the \textbf{mixing time}. Our goal is to study the order of $t^N_{\text{\normalfont mix}}$ when $N$ goes to infinity. \\
Before we come to the main results, we give some remarks on the notation. We will write $[N]$ instead of $\{ 1,\dots, N\}$ for $N \in \N$. For asymptotic estimates, we use the Landau notation with respect to $N$. For functions $f,g \colon \mathbb{N} \rightarrow \mathbb{R}$ we have that
\begin{equation}\label{landau1}
f=\mathcal{O}(g) \ \  \Leftrightarrow \ \ \exists c>0 \text{ s.t. } \limsup_{N \rightarrow \infty} \abs{\frac{f(N)}{g(N)}} \leq c \phantom{ \ }
\end{equation} and
\begin{equation}\label{landau2}
f=\Omega(g) \ \  \Leftrightarrow \ \ \exists c>0 \text{ s.t. } \liminf_{N \rightarrow \infty} \abs{\frac{f(N)}{g(N)}} \geq  c \ .
\end{equation} We write $f=\Theta(g)$ if and only if $f=\Omega(g)$ and $f=\mathcal{O}(g)$ holds. Moreover, we have that $f=o(g)$ if \eqref{landau1} is satisfied for all $c>0$. We say that an asymptotic estimate holds with high probability if for some fixed $c>0$, the respective inequality in \eqref{landau1} or \eqref{landau2} holds with probability tending to one.

\subsection{Main results}

For the rest of this article, assume that the number of particles $k=k(N)$ satisfies
\begin{equation*}
0 < \liminf_{N\rightarrow \infty}\frac{k(N)}{N} \leq \limsup_{N\rightarrow \infty}\frac{k(N)}{N} < 1 \ .
\end{equation*} We present now our main results on the mixing time of the simple exclusion process in a random environment.

 
\begin{theorem}\label{main} Let $(\eta_t)_{t\geq 0}$ denote the simple exclusion process in random environment $\omega$ with state space $\Omega_{N,k}$ and mixing time $t^N_{\text{\normalfont mix}}$. Further, assume that \begin{equation}
\E\left[ \frac{1-\omega(1)}{\omega(1)}\right] < 1 \label{ballistic}
\end{equation} holds, i.e. we are in the ballistic regime for a random walk in random environment with a drift to the right-hand side, see \cite{rwre}. We distinguish three different cases:
\begin{itemize}
\item[(i)]{ \textbf{Non-nestling case}: If we have that
\begin{equation}\label{mainnon} \P\left(\omega(1) \geq \frac{1}{2}+ \varepsilon\right) = 1
\end{equation} for some $\varepsilon>0$, then it holds that $t^N_{\text{\normalfont mix}}= \Theta(N)$ for almost all environments.}
\item[(ii)]{\textbf{Marginal nestling case}: Assume that
\begin{equation}\label{mainmargin} \P\left(\omega(1) \geq \frac{1}{2}\right) = 1 
\end{equation}
holds, but \eqref{mainnon} is not satisfied.
\begin{itemize}
\item[(a)]  There exists a function $f\colon \N \rightarrow \R^{+}$ with $\lim\limits_{N \rightarrow \infty} f(N) = \infty$ such that
\begin{equation}
\lim_{N \rightarrow \infty} \P\left( t^N_{\text{\normalfont mix}} \leq f(N)N \right) = 0
\end{equation} holds. If we have in addition that \begin{equation}\label{mainmarginmass} \P\left(\omega(1) = \frac{1}{2}\right) > 0  \end{equation} holds, then $f$ can be chosen to be in $\Theta(\log(N))$.
\item[(b)] For all environments which satisfy \eqref{mainmargin}, we have that  $t^N_{\text{\normalfont mix}}= \mathcal{O}(N\log^3(N))$ holds with high probability.
\end{itemize}}
\item[(iii)]{ \textbf{Plain nestling case}: If we have that
\begin{equation}\label{mainpure} \P\left(\omega(1) < \frac{1}{2} \right) > 0 
\end{equation} then it holds with high probability that $t^N_{\text{\normalfont mix}}= \Omega(N^{1+\delta})$ for some $\delta>0$ depending only on the environment distribution.}
\end{itemize}
\end{theorem}

\subsection{Related work}\label{relatedsection}

Mixing times for particle systems were intensively studied during the last decades. For the simple exclusion process in homogeneous environments, i.e. if \begin{equation*}
\P\left( \omega(1) = p \right) =1
\end{equation*} holds for some constant $p \in [0,1]$, mixing properties are well-understood. For $p=\frac{1}{2}$, we obtain the symmetric simple exclusion process. In 2001, Wilson introduced his famous lower bound technique which allowed him to estimate the mixing time, which is of order $N^2\log(\min(k,N-k) )$, within a factor of $2$ provided that $\lim_{N \rightarrow \infty} \min(k,N-k) = \infty$, see \cite{wilson2004}. 
His proof is based on an explicit calculation of the spectral gap and the corresponding eigenfunction of the transition matrix using Fourier analysis.
Lacoin proved that Wilson's presented lower bound is tight \cite{lacoin2016}. Moreover, he showed that the cutoff phenomenon, a sharp transition in the distance from the stationary distribution, occurs whenever $\lim_{N \rightarrow \infty} \min(k,N-k) = \infty$ holds.  \\

For $p \neq \frac{1}{2}$, we refer to the resulting process as asymmetric simple exclusion process (ASEP). Benjamini et al.\ studied the mixing time of the ASEP in the context of biased card shuffling and showed that it is asymptotically of order $N$ for any number of particles $k \in [N-1]$, see \cite{ASEP2005}. We will follow their approach for an upper bound on the mixing time in the marginal nestling case. Labbé and Lacoin \cite{labbe2016} refined this bound for the ASEP and proved cutoff. 
The proof relies on an explicit calculation of the spectral gap which was independently obtained by
Levin and Peres \cite{levin2016mixing}. They studied the mixing time of the ASEP when the bias vanishes for $N$ going to infinity. Recently, cutoff results were established in this regime by Labbé and Lacoin \cite{lacoin2018}. \\

For general (edge-)weighted graphs $G=(V,E)$, one can consider the exclusion process in which particles jump to a neighbor site according to the rates given by the weights on the corresponding edge. This is known as the varying speed model of the exclusion process. Oliveira showed that the mixing time of the exclusion process is in $\mathcal{O}\left(t_{\text{mix}}^R \cdot \log\left(\abs{V}\right)\right)$, where $t_{\text{mix}}^R$ denotes the mixing time of the random walk on $G$ \cite{oliveira2013}. Recently, this result was improved by Hermon and Pymar \cite{hermon2018}. \\

In this article, we investigate the case of one-dimensional i.i.d. environments in which the particles move at a constant speed. For a single particle, this is the classical model of a random walk in random environment, which was studied by  Kesten, Kozlov, Solomon, Spitzer  among others, see \cite{kesten1975, solomon1975,rwre}. If in addition condition \eqref{ballistic} holds, the resulting process is called the random walk in ballistic random environment. In this case, Gantert and Kochler proved that the mixing time is with high probability linear in the size of the underlying segment. Moreover, they showed that cutoff occurs \cite{Gantert2012}. \\

For the simple exclusion process in ballistic random environment, results on the mixing time were so far only available in the non-nestling case. For any (deterministic) environment with a uniform bound on the drift, Miracle and Streib showed that the mixing time is linear in the size of the segment using path coupling, see \cite{kclasses}.

\subsection{Outline of the paper}

This paper is organized as follows. In the remainder of the first section we give an outlook on open problems on mixing times of the simple exclusion process. 
The different parts of our main result will be shown in Sections \ref{canonicalsection} to \ref{uppermarginalsection}. In the second section, we exploit the structure of the state space $\Omega_{N,k}$ and define the canonical coupling for the exclusion process in non-homogeneous environments. This allows us to directly deduce Theorem \ref{main} (i) in Section \ref{nonnestlingsection}. In Section \ref{thirdsection}, the lower bounds on the mixing time for nestling environments are established. As a key step, we identify an area of small drift and then use a comparison to the boundary driven exclusion process. A corresponding upper bound is shown in the following section. Using the censoring inequality, we control the particle speed within the simple exclusion process. We then reduce the statement on the upper bound of the mixing time to a hitting time estimate, which we solve recursively. 

\subsection{Open problems}

In this article, we give bounds on the mixing time for the simple exclusion process in ballistic random environment. However, we can only give the correct order of the mixing time in the non-nestling case.
\begin{question} What are the correct orders of the mixing time in the ballistic regime?
\end{question}
In Remark \ref{remark:exclusion}, we will point out that the presented methods lead to lower bounds of order at most $N^{\frac{3}{2}}$. 
\begin{conjecture} \label{conjecture} For any ballistic random environment, the mixing time of the respective exclusion process is with high probability at most of order $N^{\frac{3}{2}}$. 
\end{conjecture}
For the random walk in random environment, Gantert and Kochler showed the cutoff phenomenon in the ballistic regime \cite{Gantert2012}. 
\begin{question} Does the exclusion process in the ballistic regime exhibit cutoff?
\end{question}
In the varying speed model for weighted graphs $G=(V,E)$, the mixing time of the exclusion process differs from the mixing time of the random walk on $G$ by at most a factor of order $\log(|V|)$, see \cite{oliveira2013}. Theorem \ref{main} (iii) shows that this relation does in general not hold in our model, the constant speed model of the exclusion process.
\begin{question} Does a similar relation as shown by Oliveira in \cite{oliveira2013} hold for the constant speed model of the exclusion process?
\end{question} 

\section{Canonical coupling for the exclusion process}\label{canonicalsection}

We now want to exploit the structure of the state space $\Omega_{N,k}$. We define a partial order on $\Omega_{N.k}$ by 
\begin{equation}\label{partialorder}
\eta \preceq \zeta  \ \ \Leftrightarrow \ \ \sum_{j=1}^J \eta(j)  \leq \sum_{j=1}^J \zeta(j) \ \text{ for all } J \in [N]
\end{equation}
for configurations $\eta,\zeta \in \Omega_{N,k}$. In other words, we say that $\eta \preceq \zeta$ if and only if the $i^{\text{th}}$ particle in $\eta$ is to the right of the $i^{\text{th}}$ particle in $\zeta$ for all $i \in [k]$ where the particles are counted from the left-hand side to the right-hand side. Observe that for any $k,N \in \N$ with $k \in [N-1]$, we have unique minimal and maximal elements $\vartheta_{N,k}$ and $\theta_{N,k}$ on the state space $\Omega_{N,k}$ given by 
\begin{align}
\vartheta_{N,k}(i) &:= \mathds{1}_{\{i > N-k\}} \label{lowerconfig}\\
\theta_{N,k}(i) &:= \mathds{1}_{\{i \leq k\}}\label{upperconfig}
\end{align} for all $i \in [N]$. 
We call $\vartheta_{N,k}$ the \textbf{ground state} on $\Omega_{N,k}$. This terminology refers to $\vartheta_{N,k}$ being the unique state of minimal energy in the potential associated to the environment in the non-nestling case.
For homogeneous environments, it is straightforward to give a grand coupling which respects this partial order on $\Omega_{N,k}$. For general environments, we will now define such a coupling which is inspired by the ideas of Lacoin's proof of the cutoff phenomenon for the symmetric simple exclusion process, see \cite[Section 8.1]{lacoin2016}. We call this the \textbf{canonical coupling} of the simple exclusion process and denote by $\mathbf{P}$ the associated probability measure. The canonical coupling will be defined with respect to a common space for all initial conditions and all environments.
It will be monotone for the partial order on $\Omega_{N,k}$ as well as for the partial order on the set of all possible environments given by 
\begin{equation}\label{environmentorder}
\omega \preceq \bar{\omega} \ \Leftrightarrow \ 1-\omega(x)\leq 1-\bar{\omega}(x) \ \forall x \in [N]
\end{equation} for environments $\omega$ and $\bar{\omega}$. Constructively, we obtain a pair $(\eta_t,\zeta_t)_{t \geq 0}$ in the canonical coupling of the exclusion processes $(\eta_t)_{t \geq 0}$ and $(\zeta_t)_{t \geq 0}$ in environments $\omega$ and $\bar{\omega}$, respectively, as follows: \\

\mbox{\parbox[l]{145mm}{Place exponential-2-clocks on all vertices $x \in [N]$.  Whenever a clock rings at a site $x$ at time $t$, we flip a fair coin and sample a Uniform-$[0,1]$ random variable $U$ independently. If the coin shows HEAD and $x\neq N$, then we proceed according to $U$ as follows: \\

If $U\leq \omega(x)$ and $\eta_t(x)=1-\eta_t(x+1)=1$ hold, we move the particle from site $x$ to site $x+1$ in $\eta_t$. If $U\leq \bar{\omega}(x)$ as well as $\zeta_t(x)=1-\zeta_t(x+1)=1$ are satisfied, then  move the particle from site $x$ to site $x+1$ in configuration $\zeta_t$. \\

We apply the following update rule when the coin shows TAIL and $x\neq 1$: \\

 If $U> \omega(x)$ and $\eta_t(x)=1-\eta_t(x-1)=1$ hold, we move the particle from site $x$ to site $x-1$ in $\eta_t$. If $U> \bar{\omega}(x)$ as well as $\zeta_t(x)=1-\zeta_t(x-1)=1$ are satisfied, then move the particle from site $x$ to site $x-1$ in configuration $\zeta_t$. \\ 
 
If none of the rules applies, we leave the current configuration unchanged. 
\\
}
} 
The following lemma is immediate from the construction of the canonical coupling.
\begin{lemma} \label{lemma2} Let $(\eta_t)_{t \geq 0}$ and $(\zeta_t)_{t \geq 0}$ be exclusion processes in environments $\omega$ and $\bar{\omega}$, respectively, according to the canonical coupling. If $\eta_0 \preceq \zeta_0$ and $\omega\preceq \bar{\omega}$, it holds that
\begin{equation*}
\mathbf{P}\left( \eta_t \preceq \zeta_t \ \forall t \geq 0 \right) = 1\ .
\end{equation*}
\end{lemma}


\section{Mixing times for non-nestling environments}\label{nonnestlingsection}

In order to show Theorem \ref{main} (i), we compare the exclusion process in random environment $\omega$ to an exclusion process in constant environment using Lemma \ref{lemma2}. We define the \textbf{hitting time} $ \tau_{\vartheta_{N,k}}$ of the ground state $\vartheta_{N,k}$ for an exclusion process $(\eta_t)_{t\geq 0}$ to be 
\begin{equation}\label{hittingtime}
 \tau_{\vartheta_{N,k}} := \inf\left( t\geq 0 \colon \eta_t =  \vartheta_{N,k}\right) \ .
\end{equation} 
The hitting time is related to the mixing time as follows.
\begin{proposition}\label{hitcorollary}  Let $t^N_{\text{\normalfont mix}}$ and $\tau_{\vartheta_{N,k}}$ denote the mixing time and the hitting time of the exclusion process $(\eta_t)_{t \geq 0}$ in environment $\omega$, respectively. If 
\begin{equation} \label{hitcondprop}
P_{\omega,\theta_{N,k}}\left(\tau_{\vartheta_{N,k}} \geq s \right) \leq \frac{1}{4}
\end{equation} holds for some $s\geq 0$, then $t^N_{\text{\normalfont mix}}\leq s$.
\end{proposition}
\begin{proof}  Since the states $\vartheta_{N,k}$ and $\theta_{N,k}$ are extremal with respect to a monotone coupling, the hitting time serves as a bound for the coupling time of all initial states. Hence, the statement follows from Corollary 5.5 of \cite{MCMT} which allows us to control the mixing time in terms of the coupling time.
\end{proof}
\begin{proof}[Proof of Theorem \ref{main} (i)]
 Let $(\eta_t)_{t\geq 0}$ denote the simple exclusion process in  environment $\omega$, where $\omega$ satisfies  \eqref{mainnon}. Let $(\zeta_t)_{t\geq 0}$ be the exclusion process with respect to a constant environment $\bar{\omega}$ given by
\begin{equation*}
\P\Big(\bar{\omega}(x)=\frac{1}{2}+\varepsilon\Big) = 1
\end{equation*} for all $x \in [N]$ and $\varepsilon>0$ taken from assumption \eqref{mainnon}. 
Benjamini et al.\ showed that the hitting time $\tau_{\vartheta_{N,k}}$ for the process $(\zeta_t)_{t\geq 0}$ satisfies
\begin{equation}\label{hitcond2}
P_{\bar{\omega}, \theta_{N,k}}\left(\tau_{\vartheta_{N,k}}  > CN \right) \leq \frac{1}{4}
\end{equation}
for all $N \in \N$ and some $C>0$ depending only on $\varepsilon$, see \cite[Theorem 1.9]{ASEP2005}. From Lemma \ref{lemma2}, we obtain that the process $(\eta_t)_{t\geq 0}$ is almost surely dominated by $(\zeta_t)_{t\geq 0}$ when we start both processes from configuration $\theta_{N,k}$. Hence, for almost all environments $\omega$, the process $(\eta_t)_{t\geq 0}$ satisfies assumption \eqref{hitcondprop} of Proposition \ref{hitcorollary} for $s = CN$ and we obtain the desired upper bound. It is straightforward to verify that a corresponding lower bound of order $N$ holds almost surely, e.g. one can consider the position of the rightmost particle in the exclusion process with initial configuration $\theta_{N,k}$ and compare it to the position of the rightmost particle in equilibrium. 
\end{proof}

\section{Lower bounds for nestling environments}\label{thirdsection}

In order to show the lower bounds in (ii) and (iii) of Theorem \ref{main}, we first study the stationary distribution for the exclusion process in more detail. In Section \ref{specialmarginproof}, we prove a lower bound of order $N\log(N)$ for environments where $\P\left( \omega(1) \leq \frac{1}{2} \right) > 0$. 
We extend this result in Section \ref{mainmarginproof} for marginal nestling and in Section \ref{mainplainproof} for plain nestling environments. 
In all parts, we assume that \eqref{ballistic} holds for all environments as well as that $2k \leq N$ as we exchange the roles of particles and empty sites otherwise. 

\subsection{Stationary distribution for ballistic random environments}

In this section, we investigate the stationary distribution for the exclusion process in a ballistic random environment.
For a configuration $\eta\in \Omega_{N,k}$, let $z_i$ denote the position of the $i^\text{th}$ leftmost particle. Whenever $\omega(x)\in (0,1)$ for all $x\in [N]$ in a fixed environment $\omega$, we can check reversibility to see that the respective stationary distribution for the exclusion process in $\omega$ is given by
\begin{equation}\label{stationaryform}
\pi_\omega^N(\eta)=\frac{1}{Z}\prod_{i=1}^{k}\prod_{x=1}^{z_i-1}\frac{\omega(x)}{1-\omega(x+1)} \ ,
\end{equation}
where $Z$ is a normalizing constant. Using this explicit form of the stationary distribution, Lemma \ref{lemmasetA} provides a set of configurations which has with high probability (with respect to the environment law $\mathbb{P}$) an exponentially small probability under the stationary distribution $\pi_{\omega}^N$. For the proof, we follow the arguments which were used to show Proposition 11 in \cite{levin2016mixing}. \\

\begin{lemma}\label{lemmasetA}
Let $(\eta_t)_{t \geq 0}$ denote the exclusion process in ballistic random environment $\omega$ with stationary distribution $\pi_\omega^N$  and define 
\begin{equation}\label{Adefinition}
A := \left\lbrace \exists  x\leq \frac{N}{4} \text{ s.t. } \eta(x)=1 \right\rbrace \ .
\end{equation} Then with high probability, we have that $\pi_\omega^N(A) \leq e^{-c N}$ holds for some $c >0$ not depending on $N$.
\end{lemma}

\begin{proof} If the environment distribution has an atom at $1$, the statement follows immediately from the observation that the event
\begin{equation*}
B=\left\lbrace \exists  x \in \left[\frac{N}{4},\frac{N}{2} \right] \colon \omega(x)=1 \right\rbrace
\end{equation*} occurs with high probability. Whenever $B$ occurs, we have that $\pi_\omega^N(A)=0$. \\

Suppose that there is no atom at $1$ and hence the stationary distribution of $(\eta_t)_{t \geq 0}$ has the form given in \eqref{stationaryform} for almost every $\omega$. Define $L(\eta)$ and $R(\eta)$ to be the leftmost particle and rightmost empty site of a configuration $\eta$, respectively. Further, we set 
\begin{equation*}
\mathcal{X}_{j,l}:=\lbrace \eta \colon L(\eta)=j, R(\eta)=l \rbrace
\end{equation*} for $j,l \in [N]$. 
We define a function $T \colon \mathcal{X}_{j,l} \rightarrow \Omega_{N,k}$ which maps $\eta \in \mathcal{X}_{j,l}$ to the configuration $T(\eta)$, where we obtain $T(\eta)$ from $\eta$ by moving the particle from position $L(\eta)$ to position $R(\eta)$. 
Using \eqref{stationaryform} as well as that $T$ is injective, we get that for $j<l$
\begin{equation*}
\pi_\omega^N(\mathcal{X}_{j,l}) = \left(\prod_{x=j+1}^{l} \frac{1-\omega(x+1)}{\omega(x)}\right) \sum_{\eta \in \mathcal{X}_{j,l}}\pi_\omega^N(T(\eta)) \leq \left(\prod_{x=j+1}^{l} \frac{1-\omega(x+1)}{\omega(x)}\right).
\end{equation*}
Note that $R(\eta)\geq N/2$ holds for all $\eta \in \Omega_{N,k}$ since $2k \leq N$. Moreover, $L(\eta) \leq N/4$ is satisfied for all $\eta \in A$ by the definition of the event $A$. We conclude that for almost all environments $\omega$
\begin{equation}\label{lastLemma1}
\pi_\omega^N(A) \leq \sum_{j \leq \frac{N}{4}, l\geq \frac{N}{2}} \pi_\omega^N(\mathcal{X}_{j,l}) \leq N^2 \cdot \max_{j \leq \frac{N}{4}, l \geq \frac{N}{2} }\left\{\prod_{x=j+1}^{l} \frac{1-\omega(x+1)}{\omega(x)}\right\}  
\end{equation} holds.  Note that the right-hand side of \eqref{lastLemma1} is with high probability exponentially decreasing in $N$. This finishes the proof of Lemma \ref{lemmasetA}.
\end{proof}
Our strategy for providing lower bounds will be the same in all three remaining parts of Section \ref{thirdsection}. We give a time $t\geq 0$ depending on $N$ such that 
\begin{equation} \label{tcondition}
P_{\omega,\lambda}(\eta_t \in A) \geq \frac{1}{2}
\end{equation}
holds with high probability for some initial distribution $\lambda$ of $(\eta_t)_{t \geq 0}$. Together with Lemma \ref{lemmasetA}, we see that with high probability
\begin{equation*}
\TV{P_{\omega,\lambda}\left( \eta_t \in \cdot \ \right) - \pi_\omega^N} \geq P_{\omega,\lambda}\left( \eta_t \in A  \right) - \pi_\omega^N(A) > \frac{1}{4}
\end{equation*} is satisfied by all $N$ large enough.

\subsection{Proof for environments with sites of non-positive drift}\label{specialmarginproof}

In this section, we consider the exclusion process in ballistic random environment for which the respective environment distribution $\mathbb{P}$ satisfies
\begin{equation}\label{environmentcondition}
\P\left( \omega(1) \leq \frac{1}{2} \right) = \alpha 
\end{equation} for some $\alpha>0$, i.e. with positive probability we have sites with zero or negative drift. Note that this includes the case of plain nestling environments as well as marginal nestling environments which in addition satisfy assumption \eqref{mainmarginmass}. The following proposition is the main result of this section.

\begin{proposition}\label{keylemma} Let $t^N_{\text{\normalfont{mix}}}$ denote the mixing time of an exclusion process $(\eta_t)_{t \geq 0}$ in environment $\omega$ where the environment distribution satisfies \eqref{environmentcondition}. Then with high probability, we have that $t^N_{\text{\normalfont{mix}}}= \Omega(N\log(N))$ holds.
\end{proposition}
In order to show Proposition \ref{keylemma}, we proceed as follows. First, we define a modified exclusion process for which it suffices to verify that condition \eqref{tcondition} holds with high probability.
We then introduce the boundary driven exclusion process and state some of its well-known properties. Moreover, we provide a coupling to the modified exclusion process on a subinterval of the line segment. We show that the subinterval can be chosen in such a way that it acts as a barrier and so with high probability, the modified exclusion process remains in $A$, defined in \eqref{Adefinition}, for a time of order $N\log(N)$. \\

For a fixed environment $\omega \in [0,1]^{N}$, let $\tilde{\omega}$ be the environment given by
\begin{equation}\label{definingenvironment}
\tilde{\omega}(i)= \frac{1}{2} \mathds{1}_{\left\{\omega(i)\leq \frac{1}{2}\right\}}+ \mathds{1}_{\left\{\omega(i) > \frac{1}{2}\right\}}
\end{equation} for all $i \in [N]$ and note that the law of $\tilde{\omega}$ satisfies \eqref{ballistic}. Moreover, for every environment $\omega$, we fix two distinct sites $x_{\omega},y_{\omega} \in [N]$ on the line segment which satisfy $x_{\omega}<y_{\omega}$.   For the exclusion process $(\eta_t)_{t \geq 0}$ in environment $\omega$, we define the corresponding \textbf{modified exclusion process} $(\xi_t)_{t\geq 0}$ with respect to $\omega$ to be the interacting particle system on $\Omega_{N,k}$ with the following transition rules. $(\xi_t)_{t\geq 0}$ obeys the same transitions as an exclusion process in environment $\tilde{\omega}$, but with the following three exceptions: 
\begin{itemize}
\item[1.] If $x_{\omega}$ is not occupied and the clock of the rightmost particle on the left-hand side of $x_{\omega}$ rings, then move it to position $x_{\omega}$.
\item[2.] At $y_{\omega}$ particles move to the left at rate $1-\tilde{\omega}(y)$ and are set to the rightmost empty site at rate $1$.
\item[3.] All particle moves from site $y_{\omega}+1$ to the left are suppressed.
\end{itemize}
Let $\tilde{P}_{\omega,\lambda}$ denote the quenched law of $(\xi_t)_{t\geq 0}$ with respect to $\omega$ and initial distribution $\lambda$. For a suitable coupling of $(\eta_t)_{t\geq 0}$ and $(\xi_t)_{t\geq 0}$ with identical initial conditions, one has that $\eta_t \succeq \xi_t$ holds for all $t\geq 0$. Since $A$ is an increasing event, we conclude that
\begin{equation}\label{Ycomp}
P_{\omega,\lambda}\left( \eta_t \in A  \right) \geq \tilde{P}_{\omega,\lambda}\left( \xi_t \in A  \right) 
\end{equation} holds for almost every environment $\omega$ and initial distribution $\lambda$. Hence, it suffices to show that 
the right-hand side of \eqref{Ycomp} is with high probability larger than $\frac{1}{2}$ for some initial distribution $\lambda$ and $t=\Omega(N\log(N))$. \\

In order to analyze the modified exclusion process, we introduce the \textbf{boundary driven symmetric simple exclusion process}  which is the Markov process $(\sigma_t)_{t \geq 0}$ with state space $\lbrace 0,1 \rbrace^M$ and generator
\begin{align}\label{boundarygenerator}
\mathcal{A}f(\sigma) &=  \sum_{x=1}^{M-1} \frac{1}{2} \left[ f(\sigma^{x,x+1})-f(\sigma) \right] \nonumber \\
 &+(1-\sigma(1))  \left[f(\sigma^1)-f(\sigma) \right] +  \sigma(M) \left[f(\sigma^M)-f(\sigma) \right]  \phantom{\sum_n^t} 
\end{align} where $\sigma^{i}$ denotes the configuration in which we flip the value of configuration $\sigma$ at position $i \in [M]$.
Intuitively, the particles perform independent symmetric random walks with an exclusion constraint on the segment of size $M$.
Moreover, particles are generated at rate $1$ at site $1$ and annihilated at rate $1$ at site $M$.  Note that $(\sigma_t)_{t\geq 0}$ forms an irreducible Markov process on $\lbrace 0,1 \rbrace^M$ with stationary distribution $\mu$. The following characterization of the particle density in $\mu$ is a known result, see \cite{landim2006}.
\begin{lemma}\label{lemma3}
The stationary distribution $\mu$ of the process $(\sigma_t)_{t\geq 0}$ satisfies 
\begin{equation}\label{stateddirichlet}
\E_{\mu}[\sigma(i)]=\frac{1}{M}\left( M+\frac{1}{2} -i \right)
\end{equation} for all $i \in [M]$, where $\E_{\mu}[\phantom{i}.\phantom{i}]$ denotes the expectation with respect to $\mu$.
\end{lemma}
\begin{proof}[Sketch of the proof]
Define the function $\rho \colon \lbrace 0,1,\dots,M+1 \rbrace \rightarrow\R$ to be
\begin{equation*}
\rho(x):=\begin{cases} \ \E_{\mu}[\sigma(x)]  &\ \text{ if } \ x \in [M]\\
\ 2-\E_{\mu}[\sigma(1)]  &\ \text{ if } \  x =  0\\
\  -\E_{\mu}[\sigma(M)]&\ \text{ if } \ x = M+1 \ .
\end{cases}
\end{equation*}
Using the generator $\mathcal{A}$, note that $\rho$ satisfies $(\Delta \rho)(x)=0 $ for all $x \in [M]$,
where $\Delta$ denotes the discrete Laplacian given by 
\begin{equation*}
(\Delta \rho)(x)= \rho(x+1)+\rho(x-1)-2\rho(x) \ .
\end{equation*}
Since $\rho$ is a solution to the one-dimensional Dirichlet problem with boundary conditions $1+\frac{1}{2M}$ at $x=0$ and $-\frac{1}{2M}$ at $x=M+1$,
we know that $\rho$ has the form stated in \eqref{stateddirichlet}.
\end{proof}

Let $Z_t$ denote the number of annihilated particles in vertex $M$ until time $t$, i.e.
\begin{equation}\label{Zcharacterization}
Z_t := \# \left\{ s \in (0,t] \colon \sum_{i=1}^M \sigma_{s^-}(i) > \sum_{i=1}^M \sigma_{s}(i)\right\} \ .
\end{equation}
From the characterization of the stationary distribution $\mu$ of $(\sigma_t)_{t\geq 0}$ in Lemma \ref{lemma3}, we deduce the following result about $(Z_t)_{t \geq 0}$.
\begin{lemma}\label{lemma4} 
The number of annihilated particles $(Z_t)_{t \geq 0}$ in $(\sigma_t)_{t\geq 0}$ satisfies
\begin{equation*}
 E_{\mu}[Z_t]= t \cdot \E_{\mu}[\sigma(M)] =\frac{t}{2M}
\end{equation*} for all $t\geq 0$, where $E_{\mu}[\phantom{i}.\phantom{i}]$ denotes the expectation with respect to the boundary driven symmetric simple exclusion process started from $\mu$.
\end{lemma}
\begin{proof}[Sketch of the proof] Since $\mu$ is stationary for $(\sigma_t)_{t\geq 0}$,  we have that
\begin{equation}\label{partialproof}
\frac{1}{t}E_\mu[Z_t] = \partial_s E_\mu[Z_s] |_{s = 0} 
\end{equation} holds for all $t > 0$. Using the definition of the generator $\mathcal{A}$ in \eqref{boundarygenerator} and Lemma \ref{lemma3}, we can deduce that the right-hand side of \eqref{partialproof} is equal to $\frac{1}{2M}$.
\end{proof}
We now want to relate the modified exclusion process $(\xi_t)_{t \geq 0}$ to the boundary driven symmetric simple exclusion process $(\sigma_t)_{t \geq 0}$. 
For $N \in \N$, let $M=M(N)$ be
\begin{equation*}
M=  \frac{1}{2\log(\alpha^{-1})}\log(N) 
\end{equation*} for $\alpha>0$ from equation \eqref{environmentcondition} and observe that the event
\begin{equation*}
C := \left\lbrace \exists  x\in \left[\frac{N}{8},\frac{N}{4}-(M+1)\right] \text{ s.t. } \omega(y)\leq \frac{1}{2} \text{ for all } y \in [x,x+M-1]\right\rbrace 
\end{equation*} holds with high probability. To see this, partition $[N/8,N/4]$ into disjoint intervals of length $M$. We then apply a Chernoff bound to the indicator random variables that an interval consists only of vertices $x$ which satisfy $\omega(x) \leq \frac{1}{2}$. \\

For $\omega \in C$, let $I(\omega)$ denote the leftmost interval of length $M$ in which all vertices $y \in I(\omega)$ satisfy $\omega(y) \leq \frac{1}{2}$ and choose the sites $x_{\omega}$ and $y_{\omega}$ in the definition of the modified exclusion process with respect to $\omega$ to be the endpoints of the interval $I(\omega)$. For $\omega \notin C$, the vertices $x_{\omega}$ and $y_{\omega}$ are chosen according to an arbitrary rule. 
Recall that $L(\eta)$ denotes the position of the leftmost particle for a configuration $\eta \in \Omega_{N,k}$. For a fixed environment $\omega$, we define
\begin{equation*}
\tau^{\ast} := \inf \left\lbrace t\geq 0 \colon L(\xi_{t}) \geq x_\omega  \right\rbrace
\end{equation*} to be the first time at which $(\xi_{t})_{t \geq 0}$ has no particles in the interval $[x_\omega-1]$. \\

Note that the modified exclusion process with respect to $\omega$ is constructed in such a way that up to time $\tau^{\ast}$, it has the law of a boundary driven symmetric simple exclusion process on the interval $I(\omega)$. This is formalized in the following lemma which we state without proof.
\begin{lemma}\label{lemma5} For every $\omega \in C$ and initial distribution $\lambda$ on $\Omega_{N,k}$, we find a coupling of $(\xi_t)_{t\geq 0}$ with respect to $\omega$ started from configuration $\psi$ chosen according to $\lambda$  and $(\sigma_t)_{t\geq 0}$ on $\{0,1\}^M$ with initial configuration $\psi |_{I(\omega)}$ such that 
\begin{equation*}
\mathbf{P}_{\omega,\lambda}\left(\xi_t(x_{\omega}-1+i)=\sigma_t(i) \ \text{ for all }i \in [M] \text{ and } t \leq \tau^{\ast}\right) =1 
\end{equation*} where $\mathbf{P}_{\omega,\lambda}$ denotes the probability measure associated to the coupling. \end{lemma}

\begin{proof}[ Proof of Proposition \ref{keylemma}] We claim that for all $\omega \in C$, we can choose an initial distribution $\lambda$ such that \begin{equation}\label{Acomp}
 \tilde{P}_{\omega,\lambda}\left( \xi_t \in A  \right) \geq \frac{1}{2}
\end{equation} holds for some $t \in \Theta(N\log(N))$. In all configurations according to $\lambda$, we first place $k/8$ particles on the positions in $[k/8]$. On $I(\omega)$, we let the particles be distributed according to the stationary distribution $\mu$ of $(\sigma_t)_{t \geq 0}$. Finally, we fill up the rightmost empty sites of $\left[N/2 ,N \right]$ such that we have in total $k$ particles present in the constructed configuration. \\

Observe that by the definition of $\tau^{\ast}$ and the event $A$
\begin{equation}\label{part1}
\tilde{P}_{\omega,\lambda}\left(\xi_t \in A \right) \geq \tilde{P}_{\omega,\lambda}\left(t \leq \tau^{\ast}\right) 
\end{equation} holds for all $t \geq 0$. Let $R_t$ denote the number of particles which move from vertex $y_{\omega}$ to the right in the modified exclusion process $(\xi_t)_{t \geq 0}$ until time $t$. Since we have initially $k/8$ particles at the positions in $[N/8]$, we get that
\begin{equation*}
\left\lbrace t \leq \tau^{\ast}\right\rbrace \supseteq \left\lbrace  R_t < \frac{k}{8}-M \right\rbrace 
\end{equation*}
for all $t \geq 0$. Using Lemma \ref{lemma5}, we conclude that
\begin{equation}\label{part2}
\mathbf{P}_{\omega,\lambda}\left( R_t < \frac{k}{8}-M \right) = \mathbf{P}_{\omega,\lambda}\left( R_t < \frac{k}{8}-M, \ t\leq \tau^{\ast} \right) = \mathbf{P}_{\omega,\lambda}\left( Z_t < \frac{k}{8}-M \right) 
\end{equation} where $Z_t$ is defined in \eqref{Zcharacterization}.
Combining \eqref{part1} and \eqref{part2}, we obtain that
\begin{equation*}
\tilde{P}_{\omega,\lambda}\left(\xi_t \in A \right) \geq \mathbf{P}_{\omega,\lambda}\left( Z_{t} < \frac{k}{8}-M \right)  \geq 1- \frac{E_{\mu}[Z_{t}]}{\frac{k}{8}-M} 
\end{equation*} using Markov's inequality in the last step.
Since $k=\Theta(N)$ and $\P(\omega \in C)=1-o(1)$ hold, Lemma \ref{lemma4} gives us that with high probability the inequality \eqref{Acomp} is satisfied for some $t = \Theta(NM)$. This finishes the proof of Proposition \ref{keylemma}. 
\end{proof}

\subsection{Proof for marginal nestling environments}\label{mainmarginproof}

In this section, we show that for all marginal nestling environments, we can find a function $f\colon \N \rightarrow \R$ tending to infinity such that
\begin{equation} \label{tcondition2}
P_{\omega,\lambda}(\eta_t \in A) \geq \frac{1}{2}
\end{equation}
holds with high probability for some initial distribution $\lambda$ and $t=Nf(N)$. This will give part (a) of Theorem \ref{main} (ii). 
We follow the arguments of Section \ref{specialmarginproof} and describe the necessary changes in the proof of Proposition \ref{keylemma}. \\

For general marginal nestling environments, the probability in \eqref{environmentcondition} may be zero. Hence, we will have to replace  the condition of sites having no positive drift by the condition of sites having "almost" no positive drift in our definitions. Formally, for every $N\in \N$, we fix a $c=c(N) \geq 0$ and $M=M(N) \in \N$. 
We denote by $(\tilde{\xi}_t )_{t \geq 0}$ the modified exclusion process with respect to $\omega$ where we replace the environment $\tilde{\omega}$ in \eqref{definingenvironment} by
\begin{equation}
\tilde{\omega}(i)= \left(\frac{1}{2}+c\right) \mathds{1}_{\left\{\omega(i)\leq \frac{1}{2}+c\right\}}+ \mathds{1}_{\left\{\omega(i) > \frac{1}{2}+c\right\}}
\end{equation} for all $i \in [N]$. Moreover, let $x_\omega$ and $y_{\omega}$ denote the endpoints of the leftmost interval $\tilde{I}(\omega)\subseteq [N/8,N/4-1]$ of length $M$ in which all vertices $x$ satisfy $\omega(x) \leq \frac{1}{2}+c$ and let them being chosen according to an arbitrary rule if no such interval exists. \\
Let $\left(\tilde{\sigma}_t \right)_{t \geq 0}$ be the boundary driven exclusion process on $\{0,1\}^M$ with generator
\begin{align}\label{specialgen}
\tilde{\mathcal{A}}f(\tilde{\sigma}) &=  \sum_{x=1}^{M-1} \left(\frac{1}{2}+c\right) \tilde{\sigma}(x)\left(f(\tilde{\sigma}^{x,x+1})-f(\tilde{\sigma})\right) \nonumber \\
&+ \ \sum_{x=2}^{M} \left(\frac{1}{2}-c\right) \tilde{\sigma}(x) \left(f(\tilde{\sigma}^{x,x-1})-f(\tilde{\sigma})\right)  \\
 &+(1-\tilde{\sigma}(1)) \ (f(\tilde{\sigma}^1)-f(\tilde{\sigma})) + \tilde{\sigma}(M)(f(\tilde{\sigma}^M)-f(\tilde{\sigma})) \nonumber \ .
\end{align}
The following statement is the analogue of Lemma \ref{lemma4}.

\begin{lemma}\label{lemma3star}
Let $\tilde{\mu}$ denote the stationary distribution of the boundary driven exclusion process $(\tilde{\sigma}_t)_{t\geq 0}$. We have that 
\begin{equation}
\E_{\tilde{\mu}}[\tilde{\sigma}(M)] \leq 2cM + \frac{2}{M+1}
\end{equation} holds where $\E_{\tilde{\mu}}[\phantom{i}.\phantom{i}]$ denotes the expectation with respect to $\tilde{\mu}$.
\end{lemma}
\begin{proof}
Define the function $\tilde{\rho} \colon \lbrace 0,1,\dots,M+1 \rbrace \rightarrow \R$ to be
\begin{equation*}
\tilde{\rho}(x):=\begin{cases} \ E_{\tilde{\mu}}[\tilde{\sigma}(x)]  &\ \text{ if } \ x \in [ M ] \\
\ 2-\E_{\tilde{\mu}}[\tilde{\sigma}(1)]  &\ \text{ if } \  x =  0\\
\  -\E_{\tilde{\mu}}[\tilde{\sigma}(M)]&\ \text{ if } \ x = M+1 \ .
\end{cases} 
\end{equation*} Using the definition of the generator $\tilde{\mathcal{A}}$ in \eqref{specialgen}, note that $\tilde{\rho}$ satisfies $\abs{\Delta\tilde{\rho}(x)} \leq 4c$ for all $x \in [M]$. Observe that the function $g$ given by
\begin{equation*}
g(x) = \tilde{\rho}(x)+2c \left( x^2 -(M+1)x \right)
\end{equation*} for all $x \in \{ 0, \dots, M+1 \}$ is discrete-convex and satisfies $g(0)= \tilde{\rho}(0)$ as well as $g(M+1)= \tilde{\rho}(M+1)$. Hence, we obtain that
\begin{equation*}
g(M) \leq \frac{M-1}{M}\tilde{\rho}(M+1) + \frac{1}{M+1} \tilde{\rho}(0) \ .
\end{equation*} Using that $\tilde{\rho}(M+1)=-\tilde{\rho}(M)$ and $\tilde{\rho}(0)\leq 2$, we obtain the desired result.
\end{proof}
 Let the event $\tilde{C}$ be given as
\begin{equation*}
\tilde{C} := \left\lbrace \exists  x\in \left[\frac{N}{8},\frac{N}{4}-(M+1)\right] \text{ s.t. } \omega(y)\leq \frac{1}{2}+c \text{ for all } y \in [x,x+M-1]\right\rbrace 
\end{equation*} and define 
\begin{equation*}
\tilde{\tau}^{\ast} := \inf \left\lbrace t\geq 0 \colon L(\tilde{\xi}_{t}) \geq x_\omega  \right\rbrace \ . 
\end{equation*} Similar to Lemma \ref{lemma5}, the modified exclusion process $(\tilde{\xi}_t)_{t\geq 0}$ can be related to the boundary driven exclusion process $(\tilde{\sigma}_t)_{t\geq 0}$ as follows.

\begin{lemma}\label{lemma5star} For every $\omega \in \tilde{C}$ and initial distribution $\lambda$ on $\Omega_{N,k}$, we find a coupling of $(\tilde{\xi}_t)_{t\geq 0}$ with respect to $\omega$ started from configuration $\psi$ chosen according to $\lambda$  and $(\tilde{\sigma}_t)_{t\geq 0}$ on $\{0,1\}^M$ with initial configuration $\psi |_{\tilde{I}(\omega)}$ such that 
\begin{equation*}
\mathbf{\tilde{P}}_{\omega,\lambda}\big(\tilde{\xi}_t(x_{\omega}-1+i)=\tilde{\sigma}_t(i) \ \text{ for all }i \in [M] \text{ and } t \leq \tilde{\tau}^{\ast}\big) =1 
\end{equation*} where $\mathbf{\tilde{P}}_{\omega,\lambda}$ denotes the probability measure associated to the coupling. \end{lemma}
Following the proof of Proposition \ref{keylemma}, we obtain a lower bound of order $N M$ provided that $c=o(\frac{1}{M^2})$ and $\tilde{C}$ occurs with high probability. Note that we can choose $\left(c(N)\right)_{N \in \mathbb{N}}$ to be a sequence tending to zero and satisfying \\
\begin{equation*}
\P\left( \omega(1) \leq \frac{1}{2}+c(N) \right) \geq \frac{1}{\log(N)}
\end{equation*}
for all $N \in \mathbb{N}$. Moreover, note that the event $\tilde{C}$ holds with high probability if
\begin{equation}\label{specialcondition1}
\lim_{N \rightarrow \infty} \frac{N}{M}\left(\frac{1}{\log(N)}\right)^M = \infty  \ .
\end{equation} Again, this follows from the observation that we can partition the interval $[N/8,N/4]$ into disjoint intervals of length $M$ and apply a Chernoff bound to the indicator random variables that an interval consists only of vertices $x$ satisfying $\omega(x)\leq \frac{1}{2}+c$. Both conditions in order to show a lower bound order $N M$ are met when we choose
\begin{equation*}
f(N)=M(N)= \min \left\lbrace c(N)^{-\frac{1}{3}}, \frac{\log(N)}{2 \log \log (N)}\right\rbrace 
\end{equation*} for all $N \in \N$. Since $\lim\limits_{N\rightarrow \infty}f(N)= \infty$, we obtain part (a) of Theorem \ref{main} (ii). \hfill $\square$


\subsection{Proof for plain nestling environments}\label{mainplainproof}

We now prove Theorem \ref{main} (iii). For plain nestling environments, there exist parameters $0<\beta,\gamma<1$ not depending on $N$ such that
\begin{equation} \label{plaincondition}
\P\left( \omega(1) \leq \frac{1}{2}-\gamma \right) = \beta 
\end{equation} holds. Set $c =c(N)= - \gamma$ and
\begin{equation}
M=M(N)=\tilde{\delta}\frac{\log(N)}{\log(\beta^{-1})}
\end{equation}
for all $N \in N$ and some $0<\tilde{\delta}<1$. For plain nestling environments, we consider the processes $(\tilde{\xi}_t)_{t\geq 0}$ and $(\tilde{\sigma}_t)_{t\geq 0}$ defined in Section \ref{mainmarginproof} with these choices for $c$ and $M$. Note that the coupling described in Lemma \ref{lemma5star} remains valid for negative values of $c$. Set $q :=\frac{1/2-c}{1/2+c}>1$.  Blythe et al.\ showed that 
\begin{equation*}
\E_{\tilde{\mu}}[\tilde{\sigma}(M)] = \Theta\left(q^{-\frac{M}{2}} \right)
\end{equation*}
where $\E_{\tilde{\mu}}[\phantom{i}.\phantom{i}]$ denotes the expectation with respect to the stationary distribution $\tilde{\mu}$ of $(\tilde{\sigma}_t)_{t\geq 0}$, see \cite[equation (72)]{Blythe2000}. (In fact, they consider a boundary driven exclusion process which is a factor of $\left(\frac{1}{2}+c\right)^{-1}$ faster than $(\tilde{\sigma}_t)_{t\geq 0}$, but has the same stationary distribution $\tilde{\mu}$.) 
Note that for plain nestling environments, the event $\tilde{C}$ occurs with high probability for our choices of $c$ and $M$. Applying the same arguments as for Proposition \ref{keylemma}, we obtain that
\begin{equation*}
t^N_{\text{mix}} = \Omega\left( N \cdot q^{\frac{M}{2}} \right) = \Omega\left( N^{1+ \tilde{\delta}\frac{\log(q)}{2\log(\beta^{-1})}}\right) 
\end{equation*} holds.
Choosing $\delta :=\tilde{\delta}\frac{\log(q)}{2\log(\beta^{-1})}$ gives us Theorem \ref{main} (iii). \hfill $\square$ \\

\begin{remark}\label{remark:exclusion}
Note that the parameter $\delta$ in the proof must be less than $\frac{1}{2}$. This follows from the observation that for the  parameters $q$ and $\beta$, $q < \beta^{-1}$ holds since
\begin{equation*}
1 > \E\left[ \frac{1-\omega(1)}{\omega(1)} \right] \geq \frac{1/2 - c}{1/2 + c} \cdot\beta = q \cdot \beta \ . 
\end{equation*}
Hence, the lower bound in Theorem \ref{main} (iii) can be at most of order $N^{\frac{3}{2}}$ using the presented techniques. This bound can for example be obtained when
\begin{equation*}
\P\left(\omega(1) = \frac{1}{4}\right) = 1- \P\left(\omega(1) = 1\right) = \beta
\end{equation*} for $\beta<\frac{1}{3}$ arbitrarily close to $\frac{1}{3}$. We believe that this is the best possible upper bound which holds with high probability for any ballistic environment distribution $\P$, see Conjecture \ref{conjecture}.
\end{remark}


\section{Upper bound for marginal nestling environments}\label{uppermarginalsection}

We now show the upper bound in Theorem \ref{main} (ii). For the entire section, we assume to have a marginal nestling environment.

\subsection{Road map for the proof}

In the proof, we combine various techniques and results for the simple exclusion process. Hence, we first want to give an overview of the strategy for the proof. 
\begin{itemize}
\item We establish a censoring inequality for the simple exclusion process in marginal nestling environment, see Proposition \ref{censoringinequality}. In words, this inequality says that under certain assumptions, leaving out transitions of a Markov process does not reduce the distance from stationarity. 
\item We study the speed of the particles on the segment when starting from the configuration with all particles at the left-hand side. In general, the speed will no longer be at a linear scale. However, when we extend the line segment to a larger size, say $N^2$, we can show that with high probability, the particles have traveled a distance of $N\log(N)$ until a time of order $N\log^3(N)$. This is formalized in Proposition \ref{proposition2}. For the proof, we partition the segment into boxes according to a censoring scheme such that with high probability, each box contains at most one particle at a time. The isolated particles perform independent random walks within their boxes. This allows us to control the particle movements with respect to their local equilibria simultaneously.
\item  The remaining part of the proof follows the ideas of Benjamini et al. in \cite{ASEP2005}. We extend the simple exclusion process to the integers and study the hitting time of the ground state. As a key tool, we will use the exclusion process with second class particle, see \cite[Section III.1]{liggett1999stochastic}. We get an upper bound on the hitting time which is of order $N\log^3(N)$ plus the hitting time of the ground state in a system with a different starting configuration, see Proposition \ref{proposition3}. We iterate this argument until the remaining hitting time is with high probability of order at most $N$. 
\end{itemize}

\subsection{The censoring inequality}
In order to state the censoring inequality, we introduce the following notations. We say for two probability measures $\nu$ and $\tilde{\nu}$ defined on a poset $\Gamma$ that $\nu$ \textbf{stochastically dominates} $\tilde{\nu}$ if $\int g  \ \textup{d}\tilde{\nu} \leq \int g \ \textup{d}\nu$ holds for all  increasing functions $g \colon \Gamma \rightarrow \R$ and write $\nu\succeq \tilde{\nu}$. Let $E=\{ \{n,n+1\} \colon n \in [N-1] \}$ denote the set of edges of the segment of size $N$. For the simple exclusion process, a \textbf{censoring scheme} is a deterministic càdlàg function 
\begin{equation*}
\mathcal{C} \colon \mathbb{R}_0^+ \rightarrow \mathcal{P}\left( E \right)
\end{equation*} where $ \mathcal{P}\left( E \right)$ denotes the set of all subsets of $E$. In the censored dynamics, a transition along an edge $e$ at time $t$ is performed if and only if $e \notin \mathcal{C}(t)$.
Lacoin showed that the censoring inequality holds for the symmetric simple exclusion process, see \cite{lacoin2016}. 
The following proposition extends this result to marginal nestling environments. 
\begin{proposition}\label{censoringinequality}  Let $\mathcal{C}$ be a censoring scheme for the simple exclusion process $(\eta_t)_{t\geq 0}$ in environment $\omega$ started from $\theta_{N,k}$, defined in \eqref{upperconfig}, and let $P^{\mathcal{C}}_{\omega, \theta_{N,k}}$ denote the law of the censored dynamics $(\eta^{\mathcal{C}}_t)_{t\geq 0}$ with the same initial conditions. Then the law of the censored dynamics stochastically dominates the law of the simple exclusion process, i.e.
\begin{equation*}
P^{\mathcal{C}}_{\omega, \theta_{N,k}}(\eta^{\mathcal{C}}_t \in \cdot) \succeq P_{\omega,\theta_{N,k}}(\eta_t \in \cdot)
\end{equation*} holds for all $t\geq 0$ and almost every environment $\omega$.
\end{proposition}
\begin{proof} Let $H \colon \Omega_{N,k} \rightarrow \R^{N-1}$ be a function given by
\begin{equation*}
\eta \mapsto H(\eta)=(H_{\eta}(x))_{x \in [N-1]}
\end{equation*} where
\begin{equation*}
H_{\eta}(x) := \sum_{z=1}^{x} \eta(z) - \frac{x k}{N}
\end{equation*} for all $\eta \in \Omega_{N,k}$ and $x \in [N-1]$. Note that $H$ is injective and let $H_{\ast}\pi^N_{\omega}$ denote the pushforward of $\pi^N_{\omega}$.  Moreover, for configurations $\eta \succeq \zeta$, we have that $H_{\eta}(x) \geq H_{\zeta}(x)$ holds for all $x \in [N-1]$. Using these observations, one can show that 
\begin{equation} \label{monotonesystem}
\left(\{H(\eta), \eta \in \Omega_{N,k} \}, \{ H_{\eta}(x), \eta \in \Omega_{N,k}, x \in [N-1]\} , [N-1], H_{\ast}\pi^N_{\omega}\right)
\end{equation}
is a monotone system with top configuration $\theta_{N,k}$ in the sense of \cite[Section 1.1]{Peres2013}. 
Note that the censoring of an edge $\{ n,n+1\}$ for some $n \in [N-1]$ is in one-to-one correspondence to keeping the value $H_{.}(n)$ fixed. Hence, we obtain Proposition \ref{censoringinequality} by applying Theorem 1.1 of \cite{Peres2013} for the system in \eqref{monotonesystem}.
\end{proof}

Next, we want to use the censoring inequality to give a lower bound on the speed of the particles within the simple exclusion process. In order to define the speed on a suitable scale, we will from now on consider the simple exclusion process $(\eta_t)_{t \geq 0}$ defined with respect to the line segment of size $N^2$ and $k\in [N-1]$ particles. Recall that $L(\eta)$ denotes the position of the leftmost particle in a  configuration $\eta$.

\begin{proposition}\label{proposition2} For the simple exclusion process $(\eta_t)_{t \geq 0}$ with initial configuration $\theta_{N^2,k}$, we have that with $\P$-probability at least $1-N^{-2}$
\begin{equation}\label{propositionresult}
P_{\omega,\theta_{N^2,k}}\left( L(\eta_{T_N}) \geq N \log(N) +N \right) \geq 1-\frac{2}{N^2}
\end{equation}
holds for  ${T_N}= cN\log^3(N)$, where $c>0$ is a sufficiently large constant. 
\end{proposition}

\begin{figure}
 
 \begin{center}
  \begin{tikzpicture}[scale=0.88]						   		
	
	\draw[line width= 1pt] (12.1,0.5) -- (12.1,-0.5) -- (6.5,-0.5) -- (6.5,0.5) -- (12.1,0.5);
	\draw[line width= 1pt] (0.9,0.5) -- (0.9,-0.5) -- (6.5,-0.5) -- (6.5,0.5) -- (0.9,0.5);	
	
	\draw[line width= 1pt](12.1,0.5-1.6) -- (12.1,-0.5-1.6) -- (6.5,-0.5-1.6) -- (6.5,0.5-1.6) -- (12.1,0.5-1.6);
	\draw[line width= 1pt] (0.9,0.5-1.6) -- (0.9,-0.5-1.6) -- (6.5,-0.5-1.6) -- (6.5,0.5-1.6) -- (0.9,0.5-1.6);		

	\draw[line width= 1pt] (3.7,0.5-3.2) -- (3.7,-0.5-3.2)  -- (9.3,-0.5-3.2) -- (9.3,0.5-3.2) -- (3.7,0.5-3.2);	
	\draw[line width= 1pt] (3.7,0.5-4.8) -- (3.7,-0.5-4.8)  -- (9.3,-0.5-4.8) -- (9.3,0.5-4.8) -- (3.7,0.5-4.8);	
	
	\draw[line width= 1pt] (12.1,0.5-6.4) -- (12.1,-0.5-6.4) -- (6.5,-0.5-6.4) -- (6.5,0.5-6.4) -- (12.1,0.5-6.4);
	\draw[line width= 1pt] (0.9,0.5-6.4) -- (0.9,-0.5-6.4) -- (6.5,-0.5-6.4) -- (6.5,0.5-6.4) -- (0.9,0.5-6.4);

		\draw[line width= 1pt] (12.1,0.5) -- (13.7,0.5);
		\draw[line width= 1pt] (12.1,-0.5) -- (13.7,-0.5);

		\draw[line width= 1pt] (12.1,0.5-1.6) -- (13.7,0.5-1.6);
		\draw[line width= 1pt] (12.1,-0.5-1.6) -- (13.7,-0.5-1.6);
		
		\draw[line width= 1pt] (0.9,0.5) -- (-0.7,0.5);
		\draw[line width= 1pt] (0.9,-0.5) -- (-0.7,-0.5);		
		
		\draw[line width= 1pt] (0.9,0.5-1.6) -- (-0.7,0.5-1.6);
		\draw[line width= 1pt] (0.9,-0.5-1.6) -- (-0.7,-0.5-1.6);		
		
		\draw[line width= 1pt] (9.3,0.5-3.2) -- (13.7,0.5-3.2);
		\draw[line width= 1pt] (9.3,-0.5-3.2) -- (13.7,-0.5-3.2);
		
		\draw[line width= 1pt] (6.5,0.5-3.2) -- (-0.7,0.5-3.2);
		\draw[line width= 1pt] (6.5,-0.5-3.2) -- (-0.7,-0.5-3.2);		

		\draw[line width= 1pt] (9.3,0.5-4.8) -- (13.7,0.5-4.8);
		\draw[line width= 1pt] (9.3,-0.5-4.8) -- (13.7,-0.5-4.8);
		
		\draw[line width= 1pt] (6.5,0.5-4.8) -- (-0.7,0.5-4.8);
		\draw[line width= 1pt] (6.5,-0.5-4.8) -- (-0.7,-0.5-4.8);		

		\draw[line width= 1pt] (12.1,0.5-6.4) -- (13.7,0.5-6.4);
		\draw[line width= 1pt] (12.1,-0.5-6.4) -- (13.7,-0.5-6.4);		
		
		\draw[line width= 1pt] (0.9,0.5-6.4) -- (-0.7,0.5-6.4);
		\draw[line width= 1pt] (0.9,-0.5-6.4) -- (-0.7,-0.5-6.4);	
		
\foreach \x in {0,-1.6,-3.2,-4.8,-6.4}
{	
	
	\node[] (A) at (-0.7,\x){} ;
 	\node[shape=circle,scale=1.2,draw] (B) at (0.2,\x){} ;
	\node[shape=circle,scale=1.2,draw] (C) at (1.6,\x) {};
	\node[shape=circle,scale=1.2,draw] (D) at (3,\x){} ;
 	\node[shape=circle,scale=1.2,draw] (E) at (4.4,\x){} ;
	\node[shape=circle,scale=1.2,draw] (F) at (5.8,\x) {};
	\node[shape=circle,scale=1.2,draw] (G) at (7.2,\x){} ;	
	\node[shape=circle,scale=1.2,draw] (H) at (8.6,\x) {};
	\node[shape=circle,scale=1.2,draw] (I) at (10,\x) {};
	\node[shape=circle,scale=1.2,draw] (J) at (11.4,\x) {};
	\node[shape=circle,scale=1.2,draw] (K) at (12.8,\x) {};
	\node[] (L) at (13.7,\x) {};
		
	\draw[thick,dashed] (A) to (B);		
	\draw[thick] (B) to (C);		
	\draw[thick] (C) to (D);	
	\draw[thick] (D) to (E);		
	\draw[thick] (E) to (F);		
	\draw[thick] (F) to (G);	
	\draw[thick] (G) to (H);	
	\draw[thick] (H) to (I);
	\draw[thick] (I) to (J);
	\draw[thick] (J) to (K);	
	\draw[thick,dashed] (K) to (L);	

		
}
	\node(L) at (-2.5,0){$t=0\phantom{S_{-}}$} ;  
	\node(L) at (-2.5,-1.6){$t=S_{-}\phantom{2}$} ;  
	\node(L) at (-2.5,-3.2){$t=S\phantom{2_{-}}$} ;  
	\node(L) at (-2.5,-4.8){$t=2S_{-}$} ;  
	\node(L) at (-2.5,-6.4){$t=2S_{\phantom{-}}$} ; 
				
	\node[shape=circle,fill=red] (k31) at (1.6,0){};
	\node[shape=circle,fill=red] (k41) at (8.6,0){} ;	
	\node[shape=circle,fill=red] (k41) at (12.8,0){} ;	
	\node[shape=circle,fill=red] (k31) at (0.2,-1.6){};
	\node[shape=circle,fill=red] (k41) at (5.8,-1.6){} ;		
	\node[shape=circle,fill=red] (k41) at (10,-1.6){} ;		
	\node[shape=circle,fill=red] (k31) at (0.2,-3.2){};
	\node[shape=circle,fill=red] (k41) at (5.8,-3.2){} ;		
	\node[shape=circle,fill=red] (k41) at (10,-3.2){} ;	
	\node[shape=circle,fill=red] (k31) at (3,-4.8){};
	\node[shape=circle,fill=red] (k41) at (8.6,-4.8){} ;		
	\node[shape=circle,fill=red] (k41) at (12.8,-4.8){} ;	
	\node[shape=circle,fill=red] (k31) at (3,-6.4){};
	\node[shape=circle,fill=red] (k41) at (8.6,-6.4){} ;		
	\node[shape=circle,fill=red] (k41) at (12.8,-6.4){} ;			




	
\end{tikzpicture}
\end{center}	
	\caption{\label{figureC}Illustration of the censoring scheme used in the proof of Proposition \ref{proposition2} with $U=2$. During each period $[iS,(i+1)S)$ for $i \in \mathbb{N}_0$, the particles shown in red are only allowed to move within their assigned boxes. }
\end{figure}
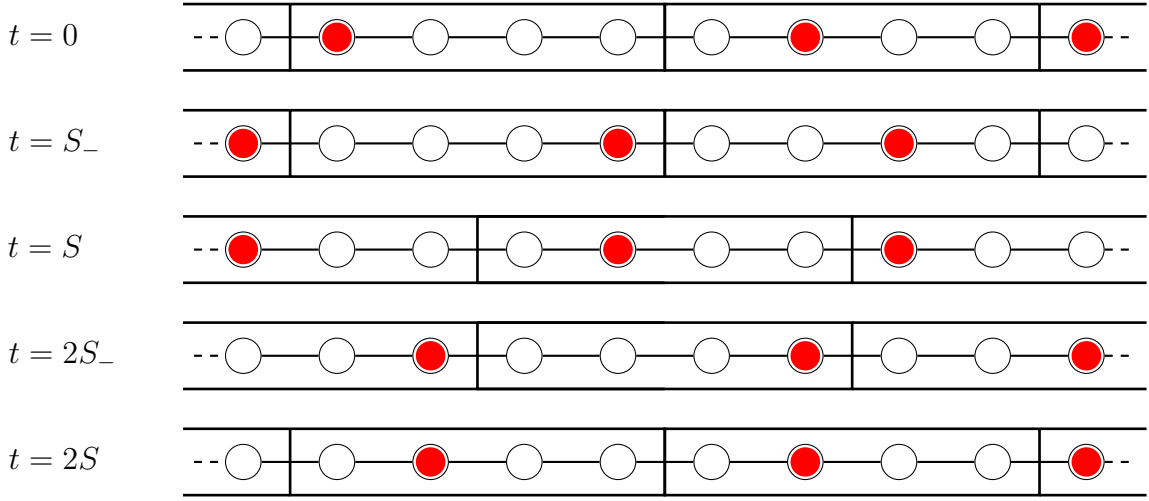

In order to show Proposition \ref{proposition2}, we provide a censoring scheme $\mathcal{C}$ for the simple exclusion process, see Figure \ref{figureC}. 
Intuitively, we alternate between two partitions of the line segment into boxes of logarithmic size. Moreover, in every second iteration, we release a particle at the left-hand side as long as there are particles available. The time between the switches of the two partitions and the size of the boxes will be chosen such that with high probability, up to time ${T_N}$ all particles move to the right half of the box within each iteration. Formally, we define $\mathcal{C}$ as follows: \\

The censoring scheme $\mathcal{C}$ remains constant within the intervals $[iS,(i+1)S)$ for all $i \in \mathbb{N}_0$ and some $S=S(N)$ which we choose later. For $i$ even, $\mathcal{C}$ contains all edges $e=\lbrace x,x+1\rbrace$ such that $x= 2j U $ for some $j\in \N$ and $x\leq N^2-2U$. Again, the value of $U=U(N)$ will be determined later on. For $i$ odd, $\mathcal{C}$ consists of all edges $e=\lbrace x,x+1\rbrace$ such that $x=(2j+1) U $ for some $j \in \mathbb{N}$ as well as $x\leq N^2-2U$. In both cases, whenever $i < 2k$, we let $e=\lbrace x,x+1\rbrace$ be the unique edge in $\mathcal{C}$ with the smallest $x$ such that
$ k-\left\lfloor \frac{i}{2} \right\rfloor \leq x $ holds. 
We remove $e$ from $\mathcal{C}$ and add the edge $\left\lbrace k-\left\lfloor \frac{i}{2}\right\rfloor -1 ,  k-\left\lfloor \frac{i}{2} \right\rfloor \right\rbrace$. This ensures that the $i^\text{th}$ particle from the right will only move from time $2(i-1)S$ onward. \\

Our goal is to control the particle movements within the boxes in the censoring scheme $\mathcal{C}$. Whenever a particle is allowed to move, it is isolated in a box of size $2U$ during an iteration (the first and last box might larger due to boundary effects but at most of size $4U$).
Consider the $i^{\text{th}}$ particle and condition on its position at time $jS$ for the largest $j$ such that $t \geq jS$ holds.
Let $B=B(i,t)$ denote the interval in which the $i^{\text{th}}$ particle may be placed at time $t\geq 0$. Further, let $C=C(i,t)$ denote the set of the rightmost $U$ vertices in $B$. Let $\mathcal{B}$ be the set of all $B(i,t)$ for some $ t \geq 0$ and $i \in [k]$. The next lemma gives an estimate on the invariant measure and the mixing time of the random walk within a box $B \in \mathcal{B}$.

\begin{lemma}\label{technicalcontroallemma}
Let $\pi_{\text{\normalfont RW}}^{\omega,B}$ denote the invariant measure of the random walk on $B\in \mathcal{B}$ in environment $\omega|_B$. There exists a constant $u>0$ such that for $U=u\log(N)$, we have with $\P$-probability at least $1-N^{-2}$ that
\begin{equation}\label{technicalstatement1}
\pi_{\text{\normalfont RW}}^{\omega,B}(C) \geq 1- N^{-5}
\end{equation} holds for all $B \in \mathcal{B}$ and $N \in \N$. For this choice of $U$, let $t^{\omega,B}_{\text{\normalfont RW}}(\varepsilon)$ denote the $\varepsilon$-mixing time of the random walk on $B\in \mathcal{B}$ in environment $\omega|_B$. There exists a constant $s>0$ such that for $S=s\log^3(N)$ and almost every environment $\omega$
\begin{equation}\label{technicalstatement2}
t^{\omega,B}_{\text{\normalfont RW}}(N^{-5}) \leq S
\end{equation} holds for all $B\in \mathcal{B}$ and $N \in \N$. Hence, with $\P$-probability at least $1-N^{-2}$, we have for all $B\in \mathcal{B}$ that a random walk started at some point in $B$ is contained in the respective set $C$ after time $S$ with probability at least $1-2N^{-5}$.
\end{lemma}
\begin{proof}
Observe that $\mathcal{B}$ contains at most $N^3$ elements by construction of the censoring scheme. 
For the random walk on $B$ the stationary distribution $\pi_{\text{\normalfont RW}}^{\omega,B}$ is given by 
\begin{equation*}
\pi_{\text{\normalfont RW}}^{\omega,B}(y) \sim \prod_{i=1}^y\frac{\omega(i)}{1-\omega(i+1)} 
\end{equation*} for all $y \in B$. Using condition \eqref{ballistic}, we know that $\E[\pi_{\text{\normalfont RW}}^{\omega,B}(y)]$ is exponentially increasing in $y$. Hence, we can choose $u>0$ such that with $\P$-probability at least $1-N^{-5}$
\begin{equation*}
\pi_{\text{\normalfont RW}}^{\omega,B}(C)\geq 1-N^{-5}
\end{equation*} holds for every $B \in \mathcal{B}$ fixed and $N \in \N$. Taking a union bound over all elements in $\mathcal{B}$ gives \eqref{technicalstatement1}. 
In order to show \eqref{technicalstatement2}, recall that $|B| \leq 4U $ holds for all $B \in \mathcal{B}$. We claim that the mixing time of the random walk in $B$ satisfies
\begin{equation*}
t^{\omega,B}_{\text{\normalfont RW}}\left(\frac{1}{4}\right) \leq 64 U^2 
\end{equation*} for all $B \in \mathcal{B}$ and almost every environment $\omega$. Using Proposition \ref{hitcorollary} for $k=1$, it suffices to give a bound on the tail of the hitting time of the rightmost site in $B$ when starting the random walk from the leftmost site in $B$. Note that this hitting time is $\P$-almost surely stochastically dominated by the respective hitting time for a symmetric simple random walk on $B$ which has mean $|B|^2$. Hence, we obtain \eqref{technicalstatement2} by using a standard estimate for the $\varepsilon$-mixing time, see \cite[equation (4.34)]{MCMT}.
\end{proof}

\begin{proof}[Proof of Proposition \ref{proposition2}] We start by making the following key observation: \\
Suppose that for all $i \in [k]$ and $j \in \N_0$ with $2(i-1)\leq j \leq {T_N}/S$, the $i^{\text{th}}$ particle (counted from the right-hand side) is contained in the set $C(i,jS)$ at time $((j+1)S)_{-}$, i.e. up to time ${T_N}$ all the particles reach the right half of their respective boxes within time $S$ whenever they are able to move. By construction of the censoring scheme $\mathcal{C}$, we then have that up to time ${T_N}$, each box contains at most one particle at a time. Moreover, each particle has moved at least $U({T_N}/S-2k)$ to the right-hand side. \\

Let $U=U(N)$ and $S=S(N)$ of Lemma \ref{technicalcontroallemma} be the size of the boxes and the time between the switches of the partitions in the censoring scheme $\mathcal{C}$, respectively. We set ${T_N}  := S(U^{-1}(N\log(N)+N)+2k)$ for all $N \in \N$. Note that we have at most $N$ particles and each particle is contained in at most $N^2$ different boxes up to time ${T_N}$ for all $N$ sufficiently large. Using Lemma \ref{technicalcontroallemma} and the key observation, we obtain that with probability at least $1-N^{-2}$
\begin{equation*}
P^{\mathcal{C}}_{\omega,\theta_{N^2,k}}\left( L(\eta_{T_N}^{\mathcal{C}}) \geq N \log(N) +N\right) \geq 1-\frac{2}{N^2} 
\end{equation*} holds.
Since the event in \eqref{propositionresult} is decreasing, we obtain the desired result by applying Proposition \ref{censoringinequality}.
\end{proof}

\subsection{Comparison to the exclusion process on the integers}

Next, we want to compare the simple exclusion process $(\eta_t)_{t \geq 0}$ on $\{ 0,1 \}^{N^2}$ to the simple exclusion process $(\eta_t^{\Z})_{t \geq 0}$ on the integers. Formally, $(\eta_t^{\Z})_{t \geq 0}$ in environment $\omega \in (0,1]^\Z$ is a Feller process with state space $\{0,1\}^{\Z}$ generated by the closure of
\begin{align}\label{generatorintegers}
\tilde{\mathcal{L}}f(\eta) &= \sum_{x \in \Z} \omega(x) \ \eta(x)(1-\eta(x+1))\left[ f(\eta^{x,x+1})-f(\eta) \right] \nonumber \\
&+ \sum_{x \in \Z} \left(1-\omega(x)\right) \ \eta(x)(1-\eta(x-1))\left[ f(\eta^{x,x-1})-f(\eta) \right]   \ .
\end{align}
Theorem 3.9 of \cite{liggett1985interacting} ensures that \eqref{generatorintegers} indeed gives rise to a Feller process.  We will use the same notation for the quenched law of $(\eta_t^{\Z})_{t \geq 0}$ as for the simple exclusion process on the segment. 
Note that the partial order as well as the canonical coupling from Section \ref{canonicalsection} naturally extend to $\Z$ when the number of particles is finite. However, we lose the existence of a unique maximal or minimal element. \\

In the following, we assume that the environment $\omega \in (0,1]^\Z$ is marginal nestling, i.e. $\lbrace \omega(x) \rbrace_{x \in\Z}$ are i.i.d. and their law satisfies conditions \eqref{ballistic} and \eqref{mainmargin}.
Let $\theta_{\Z,k}$ denote the configuration in $\{0,1\}^{\Z}$ where the particles are placed on $[k]$.
\begin{lemma}\label{lemma41} Let $(\eta_t^{\Z})_{t\geq 0}$ be the simple exclusion process on the integers in environment $\omega$ started from $\theta_{\Z,k}$. Then for all $k \in [N-1]$ with $\P$-probability at least $1-N^{-2}$
\begin{equation*}
P_{\omega,\theta_{\Z,k}}\left( L(\eta_{T_N}^{\Z}) \geq N \log(N) \right) \geq 1-\frac{4}{N^2}
\end{equation*} holds for all $N$ large enough, where ${T_N}$ is taken from Proposition \ref{proposition2}.
\end{lemma}
\begin{proof} 
For the simple exclusion process $(\eta_t^{\Z})_{t\geq0}$ on the integers in environment $\omega$, we consider its projection to the environment $\tilde{\omega}:=\omega|_{[-N,N^2-N]}$. Observe that $(\eta_t^{\Z})_{t\geq0}$ is uniquely determined by its values on $\tilde{\omega}$ whenever no particle reaches the sites $-N$ or $N^2-N$. We claim that for almost all environments $\omega$ the statements
\begin{equation}\label{event1}
P_{\omega,\theta_{\Z,k}}\left( \exists t \in [0,{T_N}]  \colon \max\left\lbrace i\geq 0 \text{ : } \eta_t^{\Z}(i)= 1 \right\rbrace \geq N^2 -N \right) \leq \frac{1}{N^2}
\end{equation}
and
\begin{equation}\label{event2}
P_{\omega,\theta_{\Z,k}}\left( \exists t \in [0,{T_N}]  \colon L( \eta_t^{\Z}) \leq -N \right) \leq \frac{1}{N^2}
\end{equation} hold for all $N$ large enough. The first statement is immediate when we consider the motion of the rightmost particle in $(\eta_t^{\Z})_{t\geq0}$. 
For the second statement, notice that the position of the left-most particle in $(\eta_t^{\Z})_{t\geq0}$ stochastically dominates the position of the left-most particle in a symmetric simple exclusion process on $\Z$ with the same initial condition. The symmetric simple exclusion process can be seen as an interchange process in which the particles swap positions along each edge independently at rate $\frac{1}{2}$. In this case, the particles perform symmetric simple random walks on $\Z$ and we can use Chernoff bounds to conclude. \\
Whenever the events in \eqref{event1} and \eqref{event2} hold, up to time ${T_N}$ the simple exclusion process $(\eta_t^{\Z})_{t\geq0}$ with initial configuration $\theta_{\Z,k}$ has on the set $[-N,N^2-N]$ the same law as a simple exclusion process $(\eta_t)_{t\geq0}$ in environment $\tilde{\omega}$ started from configuration $\theta_{N^2,k}$. Hence, Proposition \ref{proposition2} gives the desired result.
\end{proof}

Lemma \ref{lemma41} shows that the particles in $(\eta_t^{\Z})_{t\geq0}$ started from $\theta_{\Z,k}$ have passed a distance of at least $N\log(N)$ to the right-hand side until time ${T_N}$. We will now ensure that also for times larger than ${T_N}$, the particles escape fast enough. 
\begin{lemma}\label{lemma51}
For the simple exclusion process $(\eta_t^{\Z})_{t\geq0}$ in environment $\omega$ started from $\theta_{\Z,k}$, we have that for all $k \in [N-1]$ with $\P$-probability at least $1-2 N^{-2}$
\begin{equation*}
P_{\omega,\theta_{\Z,k}}\left( \forall t \geq {T_N}  \colon L( \eta_t^{\Z}) > t^{\frac{2}{3}}+N \right) \geq 1-\frac{10}{N^2}
\end{equation*} holds for all $N$ large enough with ${T_N}$ taken from Proposition \ref{proposition2}.
\end{lemma}
\begin{proof} For a given $N$, we define the sequences  $\{ N_i\}_{i \in \N}$ and $\{ t_i\}_{i \in \N}$ to be
\begin{equation*}
N_{i} := N^{\left(\left(\frac{4}{3}\right)^{i-1}\right)} \ \text{ and } \  t_i := \sum_{j=1}^{i} {T_{N_j}}
\end{equation*} for all $i \in \N$. By Lemma \ref{lemma41}, we obtain that with $\P$-probability at least $1-N_1^{-2}$
\begin{equation}\label{event6}
P_{\omega,\theta_{\Z,k}}\left(  L( \eta_{t_1}^{\Z}) \geq N_1\log(N_1) \right) \geq 1-\frac{4}{N_1^2}
\end{equation} holds. Suppose the event in \eqref{event6} occurs. Then without loss of generality, we can assume that the particles are placed on the sites in $[N_1\log(N_1),N_1\log(N_1)+k]$ at time $t_1$. Starting from this configuration, we can apply Lemma \ref{lemma41} again to obtain that with $\P$-probability at least $1-N_1^{-2}-N_2^{-2}$
\begin{equation*}
P_{\omega,\theta_{\Z,k}}\left(  L( \eta_{t_i}^{\Z}) \geq N_i\log(N_i) \text{ for } i \in \{ 1,2\} \right) \geq 1-4\left(\frac{1}{N_1^2}+\frac{1}{N_2^2}\right)
\end{equation*} holds.  Iterating this argument along the sequence $\{ N_i\}_{i \in \N}$, we see that $\P$-probability at least $1-2N^{-2}$
\begin{equation*}
P_{\omega,\theta_{\Z,k}}\left(  L( \eta_{t_i}^{\Z}) \geq N_i\log(N_i) \text{ for } i \in \N \right) \geq 1-\frac{8}{N^2}
\end{equation*} is satisfied. Observe that
\begin{equation*}
N_i\log(N_i) > (t_i)^{\frac{2}{3}}+N
\end{equation*} holds for all $i \in \N$ and $N$  large enough. Hence, it remains to consider the case of $t \in (t_i,t_{i+1})$ for some $i \in \N$. 
Using the same arguments as for the proof of \eqref{event2}, we obtain that for almost every environment $\omega$
\begin{equation*}
P_{\omega,\theta_{\Z,k}}\left( L( \eta_{t}^{\Z}) \geq N_i \ \forall t \in [t_i, t_{i+1}] \ | \ L( \eta_{t_i}^{\Z}) \geq N_i\log(N_i) \right) \geq 1-\frac{1}{N_{i+1}^2}
\end{equation*} holds for all $i \in \N$ and $N$ sufficiently large. Since we have that 
\begin{equation*}
(t_{i+1})^{\frac{2}{3}} +N < N_i
\end{equation*} holds for all $i \geq 2$ and $N$ sufficiently large, we obtain the desired result.
\end{proof}

Next, we introduce some notations for the simple exclusion process on the integers. For the simple exclusion process $(\eta_t^{\Z})_{t\geq0}$ in environment $\omega$, we let the configurations $\theta^{l,m},\vartheta^{n} \in \{0,1\}^\Z$ be given by
\begin{align*}
\theta^{l,m}(x) := \mathds{1}_{\{x \in [m]\}} + \mathds{1}_{\{x > l\}} \ \ \text{ and } \  \ 
\vartheta^{n}(x) := \mathds{1}_{\{x > n\}}
\end{align*} 
for all $x \in \Z$ and $l,m \in \N$, $n \in \Z$ with $m\leq l$. Similar to \eqref{lowerconfig} and \eqref{hittingtime}, we call $\vartheta^{n}$ a \textbf{ground state} of $(\eta_t^{\Z})_{t\geq0}$ for all $n \in \Z$ and define the \textbf{hitting time} of the ground state $\vartheta^{n}$  for $(\eta^{\Z}_t)_{t\geq 0}$ to be the random variable
\begin{equation*}
 \tau_{\vartheta^n} := \inf\left\{ t\geq 0 \colon \eta^{\Z}_t =  \vartheta^n\right\} \ .
\end{equation*} 
Moreover, for  $N \in \N$ and $k=k(N)\in [N-1]$, we define the $\mathbold{\varepsilon}$\textbf{-hitting time} of the ground state $\vartheta^{N-k}$ to be 
\begin{equation*}
t_{\text{\normalfont hit}}^{\omega,N}(\varepsilon) := \inf \left\{ t\geq 0 \colon P_{\omega,\theta^{N,k}}\left( \tau_{\vartheta^{N-k}} > t \right) \leq \varepsilon \right\}
\end{equation*} for all $\varepsilon \in (0,1)$. We now relate the $\varepsilon$-hitting time of $(\eta_t^{\Z})_{t\geq0}$ to the $\varepsilon$-mixing time $t_{\text{\normalfont mix}}^{\omega|_{[N]}}(\varepsilon)$ of the simple exclusion process $(\eta_t)_{t \geq 0}$ on the line segment of size $N$ in environment $\omega|_{[N]}$. This follows from the same arguments as Lemma 2.8 in \cite{ASEP2005}.
\begin{lemma}\label{lemma6} For almost all environments $\omega\in (0,1]^{\Z}$ and $\varepsilon >0$, we have that
\begin{equation*}
t_{\text{\normalfont mix}}^{\omega|_{[N]},N}(\varepsilon) \leq t_{\text{\normalfont hit}}^{\omega,N}(\varepsilon) \ .
\end{equation*}
\end{lemma}
\begin{proof}[Sketch of the proof] Let $(\eta_t^{\Z})_{t\geq0}$ be the simple exclusion process on the integers in environment $\omega$ with initial configuration $\theta^{N,k}$. Let $(\eta_t)_{t\geq0}$ denote the simple exclusion process on the line segment in environment $\omega|_{[N]}$ started from configuration $\theta_{N,k}$. We claim that for almost every environment $\omega \in (0,1]^{\Z}$, we can provide a coupling of the processes $(\eta_t^{\Z})_{t\geq0}$ and $(\eta_t)_{t\geq0}$ such that the hitting time $\tau_{\vartheta_{N,k}}$ of the ground state $\vartheta_{N,k}$ for $(\eta_t)_{t\geq0}$ is dominated by the hitting time $\tau_{\vartheta^{N-k}}$ of the ground state $\vartheta^{N-k}$ for $(\eta_t^{\Z})_{t\geq0}$. This can for example be achieved by using the canonical coupling on $[N]$ and making all other transitions in $(\eta_t^{\Z})_{t\geq0}$ independently. We conclude by applying a similar argument as in the proof of Proposition \ref{hitcorollary}. 
\end{proof}

In the remainder of the proof of the upper bound in Theorem \ref{main} (ii), we follow the ideas of Benjamini et al. \cite{ASEP2005}. Intuitively, we want to show that whenever the particles in the simple exclusion process on the integers have with high probability passed a distance of at least $N$ to the right-hand side, an associated exclusion process on the line segment has (almost) reached the ground state. This will be our main idea for the proof of Proposition \ref{proposition3}, which states a recursion formula for the $\varepsilon$-hitting time. For an environment $\omega \in ( 0,1]^{\Z}$ and $n \in \Z$, we denote by $\omega_n$ the environment shifted to the right-hand side by $n$, i.e.
\begin{equation}
\omega _n(x) := \omega(x-n)
\end{equation} holds  for all $x \in \Z$.
\begin{proposition} \label{proposition3} For a given $N\in \N$ and $k =k(N) \in [N-1]$, we set $N^{\prime}=N^{\frac{3}{4}}$ and $k^{\prime}= k(N^{\prime})= \frac{1}{2}N^{\frac{3}{4}}$. Consider the simple exclusion process on the integers in a marginal nestling environment $\omega$.
Set $n=N-k-N^{\prime}+k^{\prime}$ and recall ${T_N}$ from Proposition \ref{proposition2}. Then with $\P$-probability at least $1-2N^{-2}$, we have that
\begin{equation*}
t_{\text{\normalfont hit}}^{\omega,N}(\varepsilon) \leq {T_N} + t_{\text{\normalfont hit}}^{\omega_{n},N^{\prime}}\left(\varepsilon-12N^{-2}\right) 
\end{equation*} holds for all $\varepsilon>0$ and $N$ large enough.
\end{proposition}

In words, Proposition \ref{proposition3} states that the $\varepsilon$-hitting time of the ground state can with high probability be bounded from above by ${T_N}$ plus the $\varepsilon$-hitting time of the ground state $\vartheta^{k^{\prime}}$ for the simple exclusion process in the shifted environment $\omega_n$. We will see that this argument can be iterated until we reach a system where the $\varepsilon$-hitting time is with high probability of order at most $N$. \\

In order to show Proposition \ref{proposition3}, we give a brief introduction to the notion of second class particles for the exclusion process on $\Z$, see Liggett \cite[Section III.1]{liggett1999stochastic}. For a configuration $\xi \in \lbrace 0,1,2\rbrace^{\mathbb{Z}}$, we say that a vertex $x\in \mathbb{Z}$ is occupied by a \textbf{first class particle} whenever $\eta(x)=1$ and by a \textbf{second class particle} if $\eta(x)=2$ holds. We assign priorities to the vertices. Sites with first class particles get the highest priority, then vertices with second class particles and then empty sites. The \textbf{exclusion process with second class particles} on $\mathbb{Z}$ in environment $\omega$ is now given as a Feller process $(\xi_t)_{t\geq 0}$ on the state space $\lbrace 0,1,2\rbrace^{\mathbb{Z}}$ according to the following description: \\

We perform the dynamics as in the canonical coupling for the exclusion process $(\eta^{\Z}_t)_{t\geq0}$. Suppose that a vertex $x$ and its neighbor $x+1$ are chosen.  If $\xi(x)=\xi(x+1)$, we leave the configuration unchanged. Otherwise exchange the values at positions $x$ and $x+1$ in $\xi$ with probability $\omega(x)$ whenever position $x$ has a higher priority than position $x+1$. If neighbor $x-1$ is selected and position $x$ has a higher priority than position $x-1$, exchange the values with probability $1-\omega(x)$. \\

In order to reobtain a stochastic process on $\{ 0,1\}^{\Z}$, we can consider the following two projections: Let $(\xi_t^{2 \rightarrow 1})_{t\geq 0}$ be the process given by
\begin{equation*}
\xi_t^{2 \rightarrow 1} (x) := \begin{cases} 1 &\ \text{ if } \xi_t(x)\neq 0 \\ 0 &\ \text{ if } \xi_t(x)= 0\end{cases}
\end{equation*} for all $x \in \mathbb{Z}$ and $t\geq 0$. Similarly, $(\xi_t^{2 \rightarrow 0})_{t \geq 0}$ denotes the process where we have 
\begin{equation*}
\xi_t^{2 \rightarrow 0} (x) := \begin{cases} 1 &\ \text{ if } \xi_t(x)=1 \\ 0 &\ \text{ if } \xi_t(x)\neq 1\end{cases} 
\end{equation*} for all $x \in \mathbb{Z}$ and $t\geq 0$. We refer to $(\xi_t^{2 \rightarrow 1})_{t\geq 0}$ and $(\xi_t^{2 \rightarrow 0})_{t\geq 0}$ as \textbf{particle blindness} and \textbf{second class-empty site blindness}, respectively. \\
In addition, we define a third projection $(\xi_t^{\ast})_{t\geq 0}$ onto $\lbrace 0,1\rbrace^{\mathbb{Z}}$ by removing all first class particles as well as the sites corresponding to the particles and then applying projection $(\xi_t^{2 \rightarrow 1})_{t\geq 0}$. Since the resulting process is only well-defined up to translations, we initially place a tagged particle in the origin. 
For a formal description, assume that a given configuration $\xi \in \lbrace 0,1,2 \rbrace^{\mathbb{Z}}$ satisfies
\begin{equation*}
\abs{\left\lbrace i\in \mathbb{Z} \colon \xi(i) =2 \right\rbrace} = \infty \ .
\end{equation*}
Let $u$ be an enumeration of the sites without first class particles in $\xi$, where
\begin{equation*}
u(0) := \begin{cases} \inf\lbrace  i\leq 0 \colon \xi(i)=2 \rbrace & \ \text{ if } -\infty < \inf\lbrace  i\leq 0 \colon \xi(i)=2 \rbrace < +\infty \\ \inf\lbrace  i > 0 \colon \xi(i)=2 \rbrace & \ \text{ otherwise } \end{cases} \ .
\end{equation*} We obtain the positions $u(j)$ and $u(-j)$ for $j \in \mathbb{N}$ recursively by
\begin{equation*}
u(j) :=\inf\lbrace i>u(j-1) \colon \xi(i)\neq 1 \rbrace
\end{equation*}
and
\begin{equation*}
u(-j) :=\inf\lbrace i>u(-j+1) \colon \xi(i)\neq 1 \rbrace \ .
\end{equation*} We can now define $\xi^{\ast}$ as
\begin{equation}\label{projection3}
\xi^{\ast}(i) := \begin{cases}  1 &\ \text{ if } \xi(u(i))=2 \\ 0 &\ \text{ if } \xi(u(i))=0\end{cases}
\end{equation} for all $i \in \mathbb{Z}$.
In order to obtain a stochastic process $(\xi^{\ast}_t)_{t\geq 0}$, we denote by $u_t(i)$ the position of the particle at time $t$ which is in position $u(i)$ in $\xi_0$ and then apply \eqref{projection3} accordingly. The proof of Proposition \ref{proposition3} will now be an interplay of the three projections $(\xi_t^{2 \rightarrow 1})_{t\geq 0}$, $(\xi_t^{2 \rightarrow 0 })_{t\geq 0}$ and $(\xi_t^{\ast})_{t\geq 0}$ of a simple exclusion process with second class particles $(\xi_t)_{t\geq 0}$.

\begin{proof}[Proof of Proposition \ref{proposition3}] Let $(\xi_t)_{t\geq 0}$ be the simple exclusion process with second class particles in environment $\omega$ with 
\begin{equation*}
\xi_0(x) := \begin{cases} 0& \ \text{ if } x\leq 0 \\ 
1 & \ \text{ if } x \in [k] \\ 
0 & \ \text{ if } x \in [k+1,N] \\ 
2 & \ \text{ if } x> N \\ \end{cases} 
\end{equation*} as initial configuration. Observe that the process $(\xi_t^{2 \rightarrow 1})_{t \geq 0}$ has the same law as a simple exclusion process in environment $\omega$ started from configuration $\theta^{N,k}$. Our goal is to bound the hitting time of the ground state $\vartheta^{N-k}$ for the process $(\xi_t^{2 \rightarrow 1})_{t \geq 0}$. 
We make the following key observation: 
Suppose that at time $t\geq 0$, the two events
\begin{align*}
K_1 &:=\left\{ \inf \{ x \in \Z \colon \xi_t(x)=1 \} \geq t^{\frac{2}{3}}+N \right\}\\
K_2 &:= \left\{ \xi_t^{\ast}(x) = \mathds{1}_{\{ x\geq 0 \}} \ \forall x \in \Z\right\} 
\end{align*} occur. Then $\xi_t^{2 \rightarrow 1}= \vartheta^{N-k}$ holds. To see this, note that if $K_1$ occurs, then there exists a second class particle which is on the left-hand side of the leftmost first-class particle in $\xi_t$. If $K_2$ occurs, then all empty sites are placed on the left-hand side of the leftmost second class particle in $\xi_t$. We claim that with $\P$-probability at least $1-2N^{-2}$, we have that
\begin{equation}\label{k1statement}
P_{\omega,\xi_0}\left( K_1 \text{ holds for all } t \geq {T_N}\right) \geq 1-10N^{-2}
\end{equation}
holds. Note that the process $(\xi_t^{2 \rightarrow 0})_{t \geq 0}$ has the same law as a simple exclusion process in environment $\omega$ started from configuration $\theta_{\Z,k}$ and so \eqref{k1statement} follows from Lemma \ref{lemma51}. \\

We now want give an upper bound on the first time $t \geq {T_N}$ such that the event $K_2$ occurs. We claim that for almost every marginal nestling environment $\omega$
\begin{equation}\label{chernoffestimate}
P_{\omega,\xi_0}\left(  \sup\{ i\geq 0 \colon \xi_t^{\ast}=0 \} < t^{\frac{2}{3}}  \ \forall t \geq {T_N} \right) \geq 1-\frac{1}{N^2}
\end{equation} holds. Note that the position of the rightmost empty site in the process $(\xi^{\ast}_t)_{t\geq 0}$ is stochastically dominated by the position of the right-most empty site for a symmetric simple exclusion process for starting configuration $\psi \in \{0,1\}^{\Z}$ with $\psi(x) = \mathds{1}_{x \leq 0}$ for all $x \in \Z$. The symmetric simple exclusion process can be seen as an interchange process and hence, we obtain \eqref{chernoffestimate} by applying Chernoff bounds. 
Using the same argument for the position of the leftmost second class particle, we see that
\begin{equation}\label{chernoffestimate2}
P_{\omega,\xi_0}\left(  \inf\{ i\leq 0 \colon \xi_t^{\ast}=1 \} > - t^{\frac{2}{3}}  \ \forall t \geq {T_N} \right) \geq 1-\frac{1}{N^2}
\end{equation} holds for almost every environment $\omega$. Observe that if the events in \eqref{k1statement} and \eqref{chernoffestimate} hold, no first-class particle will be next to an empty site for any time $t \geq {T_N}$. Since transitions between first and second class particles do not change a configuration in $(\xi_t^{\ast})_{t\geq {T_N}}$, the process $(\xi_t^{\ast})_{t\geq {T_N}}$ then has the law of a simple exclusion process in environment $\omega_n$. Hence, the hitting time of the ground state $\vartheta^{k^{\prime}}$ in the process $(\xi_t^{\ast})_{t\geq {T_N}}$ started from $\xi^{\ast}_{T_N}$ gives an upper bound on the hitting time of the ground state $\vartheta^{N-k}$ for the process $(\xi_t^{2 \rightarrow 1})_{t \geq 0}$. 
We now show that it suffices to consider the hitting time of the ground state for $(\xi_t^{\ast})_{t\geq {T_N}}$ started from configuration $\theta^{N^{\prime},k^{\prime}}$ at time ${T_N}$. Observe that the partial order from \eqref{partialorder} extends to the set of configurations 
\begin{equation*}
A := \left\{ \eta \in \{ 0,1\}^{\Z} \colon \sum_{i=-\infty}^{k^{\prime}}\eta(i)= \sum_{i=k^{\prime}+1}^{\infty}(1-\eta(i)) < \infty \right\}
\end{equation*} i.e. for $\eta,\zeta \in A$, we have that
\begin{equation*}
\eta \preceq \zeta  \ \ \Leftrightarrow \ \ \sum_{i=-\infty}^j \eta(i)  \leq \sum_{i=-\infty}^j \zeta(i) \ \text{ for all } j \in \Z \ .
\end{equation*}
Provided that the events in \eqref{chernoffestimate} and \eqref{chernoffestimate2} occur, we have that $\xi^{\ast}_{T_N} \preceq \theta^{N^{\prime},k^{\prime}}$ holds for all $N$ sufficiently large. Note that the canonical coupling extended for the exclusion process on $\Z$ preserves the partial order on $A$. Combining these observations, we obtain that the hitting time of the ground state $\vartheta^{N-k}$ for the process $(\xi_t^{2 \rightarrow 1})_{t \geq 0}$ is stochastically dominated by the hitting time of the ground state $\vartheta^{k^{\prime}}$ in the process $(\xi_t^{\ast})_{t\geq {T_N}}$ started from $\theta^{N^{\prime},k^{\prime}}$ at time ${T_N}$ whenever the events in \eqref{k1statement},\eqref{chernoffestimate} and \eqref{chernoffestimate2} occur. This gives the desired result.
\end{proof}
\begin{remark} Note that the assumption of having a marginal nestling environment is essential in the proof of Proposition \ref{proposition3}. In the plain nestling case, the arguments in order to show \eqref{chernoffestimate} and \eqref{chernoffestimate2} are not applicable.
\end{remark}

The next lemma gives a bound on the $\varepsilon$-hitting time of the ground state when the parameters in the initial configuration of the simple exclusion process on the integers are not increasing too fast in $N$.

\begin{lemma} \label{proposition4} For all $\varepsilon>0$, we find a sequence $(M_N)_{N \in \N}$ with
 $\lim_{N \rightarrow \infty} M_N = \infty$ such that the $\varepsilon$-hitting time of the ground state $\vartheta^{M_N/2}$ for a simple exclusion process on the integers with initial condition $\theta^{M_N,M_N/2}$ satisfies
\begin{equation*}
 \P \left( t_{\text{\normalfont hit}}^{\omega,M_N}(\varepsilon) < N \right) \geq 1- \frac{1}{M(N)} 
\end{equation*} for all $N$ sufficiently large.
\end{lemma}
\begin{proof} For every $m \in \Z$, define the set of configurations
\begin{equation*}
A_m := \left\{ \eta \in \{ 0,1\}^{\Z} \colon \sum_{i=-\infty}^{m}\eta(i)= \sum_{i=m+1}^{\infty}(1-\eta(i)) < \infty \right\} 
\end{equation*} and note that $\vartheta^{m},\theta^{2m,m} \in A_m$  holds for all $m \in \Z$. Using Theorem 1.1(b) of \cite{jung2003} and Theorem B.52 of \cite{liggett1999stochastic}, we have that the exclusion process restricted to $A_m$ forms an ergodic Markov chain for almost every environment $\omega$. Hence, for all $m$ and $\varepsilon>0$ fixed, we have that the $\varepsilon$-hitting time of the ground state  $\vartheta^{m}$ for a simple exclusion process started from $\theta^{2m,m}$ satisfies
\begin{equation*}
\lim_{N \rightarrow \infty} \P \left( t_{\text{\normalfont hit}}^{\omega,m}\left( \varepsilon \right) < N\right) = 1 \ .
\end{equation*} For every $N \in \N$, we set
\begin{equation*}
M_N := \max\left\{ m \in \N \colon  \P \left(t_{\text{\normalfont hit}}^{\omega,m}\left( \varepsilon \right) < N \right) \geq 1- m^{-1}\right\}
\end{equation*}
in order to obtain a sequence $(M_N)_{N \in \N}$ as stated in Lemma \ref{proposition4}.
\end{proof}

\begin{proof}[Proof of Theorem \ref{main}(ii) part (b)]
By Lemma \ref{lemma6}, it suffices to show that
\begin{equation*}
\lim_{N\rightarrow \infty} \P \left(  t_{\text{\normalfont hit}}^{\omega,N}\left( \frac{1}{4} \right) < cN\log^3(N) \right) = 1
\end{equation*} holds for some $c >0$. For $N \in \N$ large enough and $M_N$ of Lemma \ref{proposition4} with respect to $\varepsilon=\frac{1}{8}$, define 
\begin{equation*}
I_N := \min\left\lbrace i \in \N \colon  N^{\left(\frac{3}{4}\right)^{i}} < M_N \right\rbrace \ .
\end{equation*} 
We iterate Proposition \ref{proposition3} now $I_N$ many times to obtain that with probability at least $1-4M_N^{-3/4}$
\begin{align*}
 t_{\text{\normalfont hit}}^{\omega,N} \left( \frac{1}{4} \right) &\leq \sum_{i=0}^{I_N} T_{\Big(N^{\left(\frac{3}{4}\right)^{i}}\Big)}  + t_{\text{\normalfont hit}}^{\omega_l,M_N}\left( \frac{1}{4}-\sum_{i=0}^{I_N} N^{-\left(\frac{3}{4}\right)^{i}} \right) \\
 &\leq 2 {T_N}  + t_{\text{\normalfont hit}}^{\omega_l, M_N}\left( \frac{1}{8} \right) 
\end{align*} holds for all $N$ sufficiently large and  some $l\in \Z$ depending only on $N$. Since the shifted environment $\omega_l$ has the same law as $\omega$, we conclude that with probability at least $1-5M_N^{-3/4}$  
\begin{align*}
 t_{\text{\normalfont hit}}^{\omega,N} \left( \frac{1}{4} \right) \leq 2 {T_N}  + N \leq cN\log^3(N)
\end{align*} holds for some $c>0$ and $N$ large enough. This finishes the proof of Theorem \ref{main}. \end{proof}

\textbf{Acknowledgments} I am grateful to Nina Gantert for suggesting this topic as well as for her encouragement and helpful comments. Parts of the results were established in my Master thesis at TU München written under her supervision. I like to thank Noam Berger for many inspiring discussions and for sharing his intuition about the exclusion process. Michael Prähofer is acknowledged for suggesting the Matrix product ansatz. Moreover, I want to thank two anonymous referees for a careful reading and various suggestions which significantly helped to improve the paper.
This work is supported by TopMath, the graduate program of the Elite Network of Bavaria and the graduate center of TUM Graduate School.

\bibliographystyle{plain}
\bibliography{Paper}
\end{document}